\documentclass[11pt,a4paper]{amsart}%
\usepackage{amsfonts}
\usepackage{amssymb}
\usepackage[utf8]{inputenc}
\usepackage{amsmath}
\usepackage{graphicx}%
\providecommand{\U}[1]{\protect\rule{.1in}{.1in}}
\usepackage[colorlinks=true, linkcolor=red, citecolor=blue]{hyperref}

\usepackage[]{epsfig}
\usepackage[]{pstricks}
\usepackage{tikz}

\newtheorem{theorem}{Theorem}[section]
\newtheorem{proposition}[theorem]{Proposition}

\newtheorem{remark}{Remark}
\theoremstyle{remark}

\newcommand{\remove}[1]{ }

\def\R{\mathbb R}
\def\be{\begin{equation}}
\def\ee{\end{equation}}
\def\ba{\begin{eqnarray}}
\def\ea{\end{eqnarray}}

\numberwithin{equation}{section}
\begin{document}
\title[Well-posedness and Controllability: Higher order KdV]{Well-posedness and controllability of Kawahara equation in weighted Sobolev spaces}
\author[Capistrano--Filho]{Roberto de A. Capistrano--Filho*}\thanks{*Corresponding author: roberto.capistranofilho@ufpe.br}
\address{Departamento de Matemática, Universidade Federal de Pernambuco (UFPE), 50740-545, Recife-PE, Brazil.}
\email{roberto.capistranofilho@ufpe.br}
\author[Gomes]{Milena de S. Gomes}
\address{Departamento de Matemática, Universidade Federal de Pernambuco (UFPE), 50740-545, Recife-PE, Brazil.}
\email{mili.monique90@gmail.com}

\begin{abstract}
We consider the Kawahara equation, a fifth order Korteweg-de Vries type equation, posed on a bounded interval. The first result of the article is related to the well-posedness in weighted Sobolev spaces, which one was shown using a general version of the Lax--Milgram Theorem. With respect to the control problems, we will prove two results. First, if the control region is a neighborhood of the right endpoint, an exact controllability result in weighted Sobolev spaces is established. Lastly, we show that the Kawahara equation is controllable by regions on $L^2$ Sobolev space, the so-called \textit{regional controllability}, that is, the state function is exact controlled on the left part of the complement of the control region and null controlled on the right part of the complement of the control region.
\end{abstract}
\maketitle

\section{Introduction\label{Sec0}}
\subsection{Presentation of problem }
Fifth order Korteweg-de Vries (KdV) type equation can be written as
\be\label{kaw}
u_t+u_x+\beta u_{xxx}+ \alpha u_{xxxxx}+uu_x=0,
\ee
where $u=u(t,x)$ is a real-valued function of two real variables $t$ and $x$, $\alpha$ and $\beta$ are real constants. When we consider, in \eqref{kaw},  $\beta=1$ and $\alpha=-1$, T. Kawahara  \cite{Kawahara} introduced a dispersive partial differential equation which describes one-dimensional propagation of small-amplitude long waves in various problems of fluid dynamics and plasma physics, the so-called Kawahara equation. 

In this article we shall be concerned with the well-posedness and control properties of Kawahara when the control acting through a forcing term $f$ incorporated in the equation:
\be
u_{t}+u_{x}+u_{xxx}-u_{xxxxx}+ uu_{x}=f, \quad t\in [0,T], \ x\in[0,L],
\label{a1}
\ee
with appropriate boundary conditions. Our main purpose is to see whether there are solutions in some appropriate Sobolev spaces and if one can force solutions of \eqref{a1} to have certain desired properties by choosing an appropriate control input $f$.  We will consider the following controllability issue:

\vspace{0.2cm}

\noindent\textit{Given an initial state $u_{0}$ and a
terminal state $u_{1} $ in a certain space, can one find an
appropriate control input $f$ so that the equation (\ref{a1}) admits a
solution $u$ which equals $u_{0}$ at time $t=0$ and  $u_{1}$ at time
$t=T$?}

\vspace{0.1cm}

If one can always find a control input $f$ to guide the system described by
(\ref{a1}) from any given initial state $u_{0}$ to any given terminal state
$u_{1}$, then the system (\ref{a1}) is said to be {\em exactly
controllable}. If the system can be driven, by means of a control $f$, from
any state to the origin (i.e. $u_{1}\equiv0$), then one says that system \eqref{a1} is
{\em null controllable}.

\subsection{Previous results}
Kawahara equation is a dispersive partial differential equation (PDE) describing numerous wave phenomena such as magneto-acoustic waves in a cold plasma \cite{Kakutani}, the propagation of long waves in a shallow liquid beneath an ice sheet \cite{Iguchi}, gravity waves on the surface of a heavy liquid \cite{Cui}, etc. In the literature this equation is also referred to as the fifth-order KdV equation \cite{Boyd}, or singularly perturbed KdV equation \cite{Pomeau}.

There are valuable efforts in the last years that focus on the analytic
and numerical methods for solving \eqref{kaw}. These methods
include the tanh-function method \cite{Berloff}, extended tanh-function method
\cite{Biswas}, sine-cosine method \cite{Yusufoglu}, Jacobi elliptic functions
method \cite{Hunter}, direct algebraic method \cite{Polat}, decomposition
methods \cite{Kaya}, as well as the variational iterations and homotopy
perturbations methods \cite{Jin}. For more details see \cite{Bridges,
Shuangping, Sirendaoreji, Wazwaz1, DZhang}, among others. These approaches
deal, as a rule, with soliton-like solutions obtained while one considers
problems posed on a whole real line. For numerical simulations, however, there
appears the question of cutting-off the spatial domain \cite{Bona1, Bona2}.
This motivates the detailed qualitative analysis of problems for \eqref{kaw} in bounded regions \cite{Faminskii}.

In addition to the aspects mentioned above, the Kawahara equation has been intensively studied from various other aspects of mathematics, including the  well-posedness, the existence and stability of solitary waves, the integrability, the long-time behavior, the stabilization and control problem, etc. For example, concerning the Cauchy problem in the real line, we can cite, for instance, \cite{Cui,Faminskii,kepove,ponce} and references therein for a good review of the problem. For what concerns the boundary value problem, the Kawahara equation with homogeneous boundary conditions was investigated by Doronin and Larkin \cite{Doronin2} and also in a half-strip in \cite{FaOp} for Faminkii and Opritova. Still in relation with results of well-posedness in weighted Sobolev space, we can mention \cite{khanal} and the reference therein.

We can not forget the advances in control theory for the Kawahara equation. Recently, the first author, in \cite{CaKawahara}, studied the stabilization problem and conjectured a critical set phenomenon for Kawahara equations as occurs with the KdV equation \cite{CaZh,Rosier} and Boussinesq KdV-KdV system \cite{CaPaRo1}, for example. The characterization of critical sets for the Kawahara equation is a completely open and interesting problem, we can cite for a good overview about this topic \cite{Vasconcellos}. 

It is important to note that the (third-order) Korteweg–de Vries equation has drained much attention (see in particular \cite{Bona1,Bona2,Faminskii,goubet}). With respect of the internal and boundary controllability problem the equivalent for the Korteweg–de Vries equation has also known many developments lately, see \cite{CaPaRo,cerpa,cerpa1,Cocr,GGa,GG1,Rosier} and the reference therein.

Let us mention the result proved by Glass and Guerrero, in \cite{GG}, with respect to boundary controllability of fifth order KdV equation. In this work the authors treated the exact controllability when two or five controls are inputting on the boundary conditions.  Still related to the control and stabilization problem we can cite \cite{CaKawahara,Doronin,Guzman,Vasconcellos}. By contrast, the mathematical theory pertaining to the study of the internal controllability in a bounded domain is considerably less advanced for the equation \eqref{kaw}. 

As far as we know, the control problem was, first, studied in \cite{zhang,zhang1} when the authors considered a periodic
domain $\mathbb{T}$ with a distributed control of the form
\[
f(x,t)=(Gh)(x,t):= g(x)(  h(x,t)-\int_{\mathbb{T}}g(y) h(y,t) dy),
\]
where $g\in C^\infty (\mathbb T)$ was such that $\{g>0\} = \omega$ and $\int_{\mathbb T} g(x)dx=1$, and 
the function $h$ was considered as a new control input. 

To finish this historical overview, more recently, Chen \cite{MoChen} considered the Kawahara equation posed on a bounded interval $(0,T)\times(0,L)$, with a distributed control. The author established a Carleman estimate for the Kawahara equation with an internal observation, as done in \cite{CaPaRo} for the KdV equation. Then, applying this Carleman estimate, she showed that the Kawahara equation \eqref{a1} is null controllable when $f$ is supported in a $\omega\subset(0,L)$.

In this article, we will try to close the possibilities for the internal controllability issues for the Kawahara equation. We shall consider the system 
\begin{equation}
\left\{
\begin{array}
[c]{lll}%
u_{t}+u_{x}+uu_{x}+u_{xxx}-u_{xxxxx}= f  &  & \text{in }(  0,T)  \times(  0,L)  \text{,}\\
u(  t,0)  =u(  t,L)  =u_{x}(  t,0)=u_{x}(  t,L)  =u_{xx}(  t,L)=0 &  & \text{in }(  0,T)  \text{,}\\
u(  0,x)  =u_{0}(  x)  &  & \text{in }( 0,L)  \text{.}%
\end{array}
\right.  \label{S1}%
\end{equation}
As the smoothing effect is different from those in a periodic domain, the results in this paper turn out to be very different from those in  \cite{zhang,zhang1}.  First, for a controllability result in $L^2(0,L)$, the control $f$ has to be taken in the space $L^2(0,T,H^{-2}(0,L))$. Actually, with    any control $f\in L^2(0,T,L^2(0,L))$, the solution of \eqref{S1} starting from $u_0=0$ at $t=0$ would remain in $H^2_0(0,L)$ (see \cite{GG}).  On the other hand, as for the boundary control, the localization of the distributed control plays a role in the results. It is important to point out that the results of the article, presented in the next section, remain valid for the fifth order KdV equation \eqref{kaw}.

\subsection{Main results}
The aim of this paper is to address the controllability issue for the Kawahara equation \eqref{S1} on a bounded domain with a distributed control. Our first result is the following one: 
\begin{theorem}
\label{thmB}
Let $T>0$, $\omega=(l_1,l_2)=(L-\nu ,L)$ where $0<\nu <L$.
Then, there exists $\delta>0$ such that for any $u_{0}$, $u_{1}\in L^2_{  \frac{1}{L-x}  dx}$
with
\[
\left\Vert u_{0}\right\Vert _{L^2_{ \frac{1}{L-x}  dx} } \leq\delta \ \text{ and } \  
\left\Vert u_{1}\right\Vert _{L^2_{  \frac{1}{L-x}  dx}}\leq\delta ,
\]
one can find a control input $f\in L^{2}(  0,T;H^{-2}(0,L))  $ with $\text{supp} (f)\subset (0,T) \times \omega$ 
such that, the solution of (\ref{S1}) $$u\in C^0([0,L],L^2 (0,L) )
\cap L^2(0,T,H^2(0,L))$$  satisfies  $$u(T,\cdot)  =u_{1}\,\text{ in } (0,L) \text{ and } u\in C^0([0,T],L^2_{  \frac{1}{L-x}  dx} ).$$ Additionally, 
$f\in L^2_{ (T-t) dt}(0,T,L^2(0,L))$.  
\end{theorem}

Actually, we shall have to investigate the well-posedness of the linearization of \eqref{S1} in the space $L^2_{ \frac{1}{L-x}dx}$ and  the well-posedness of the (backward) adjoint system in  the ``dual space'' $L^2_{ (L-x) dx}$. The proof of this result relies on a general version of the Lax–Milgram theorem (see, e.g., \cite{Lions}).The observability inequality is obtained by \textit{multiplier method, compactness-uniqueness arguments and a unique continuation property}. Finally, the exact controllability is extended to the nonlinear system by using the \textit{contraction mapping principle}. 

\vspace{0.2cm}

The second result of this work is devoted to prove that is possible to control the state function on $(0,l_1)$, so that a "regional controllability'' can be established:

\begin{theorem}
\label{thmC}
Let $T>0$ and $\omega = (l_1,l_2)$ with $0<l_1<l_2<L$.  Pick any number $l_1'\in (l_1,l_2)$. Then there exists a number $\delta >0$ such that  for any
$u_0,u_1\in L^2(0,L)$ satisfying 
\[
||u_0||_{L^2(0,L)} \le \delta, \qquad ||u_1||_{L^2(0,L)} \le \delta ,
\]
one can find a control $f\in L^2(0,T,H^{-2}(0,L))$ with $\text{supp}(f)\subset (0,T)\times \omega$ such that the solution  of \eqref{S1}
$$u\in C^0([0,T],L^2(0,L)) \cap L^2(0,T,H^2 (0,L))$$ satisfies 
\begin{equation*}
u(T,x) =\left\{ 
\begin{array}{ll}
u_1(x) & \text{ if }  \ x\in (0,l_1');\\
0 &\text{ if } \ x\in (l_2,L).
\end{array}
\right. 
\end{equation*}
\end{theorem}
The proof of Theorem \ref{thmC} combines \cite[Theorem 1.1]{MoChen}, a boundary controllability result from \cite{GG} and the use of a cut-off function. Note that, as for the boundary control, the internal control gives a control of hyperbolic type in the left direction and a control of parabolic type in the right direction.

Thus, our work is outlined in the following way: Section \ref{Sec1} is devoted to prove that fifth order KdV equation is well-posed in the weighted spaces $L^2_{xdx}$ and $L^2_{(L-x)^{-1}dx}$. In the Section \ref{Sec2}, our goal is to prove Theorem \ref{thmB}. Section \ref{Sec3} we will give the proof of Theorem \ref{thmC}. Finally, in the last section, Section \ref{Sec4}, we will present some additional comments and some open issues.

\section{A fifth order KdV equation in weighted Sobolev spaces \label{Sec1}}

\subsection{The linear system}
For any measurable function $w:(0,L)\to (0,+\infty )$ (not necessarily in $L^1(0,L)$), we introduce the weighted $L^2-$space
\[
L^2_{w(x)dx} =\{ u\in L^1_{loc}(0,L);\ \int_0^L u(x)^2 w(x)dx <\infty \}. 
\] 
It is a Hilbert space when endowed with the scalar product
\[
(u,v)_{L^2_{w(x)dx}}=\int_0^L u(x)v(x) w(x)dx. 
\]
We first prove the well-posedness of the following linear system
\begin{equation}
\left\{
\begin{array}
[c]{lll}%
u_{t}+u_{x}+\beta u_{xxx}+\alpha u_{xxxxx}=0 &  & \text{in }(  0,T)  \times(
0,L)  \text{,}\\
u(  t,0)  =u(  t,L)  =u_{x}(  t,0)=u_{x}(  t,L)=u_{xx}(  t,L)  =0 &  &
\text{in }(  0,T)  \text{,}\\
u(  0,x)  =u_{0}(  x)  &  & \text{in }(
0,L)  \text{,}%
\end{array}
\right.  \label{ec_lin}%
\end{equation}
in both spaces $L^2_{xdx}$ and $L^2_{\frac{1}{L-x}dx}$, where $\alpha$ and $\beta$ are real constants. We need  the following general version of the Lax–Milgram
Theorem (see, e.g., \cite{Lions}).
\begin{theorem}
\label{khanal} 
Let $W\subset V\subset H$ be three Hilbert spaces with continuous and dense embeddings. 
Let $a(v,w)$ be a bilinear form defined on $V\times W$ that satisfies the following properties:\\
(i) ({\bf Continuity})
\begin{equation*}
\label{gou1}
a(v,w) \le M ||v||_V ||w||_W, \quad \forall v\in V,\ \forall w\in W;
\end{equation*}
(ii) ({\bf Coercivity})
\begin{equation*}
\label{gou2}
a(w,w) \ge m ||w||_V^2, \quad \forall w\in W;
\end{equation*}
Then for all $f\in V'$ (the dual space of $V$), there exists $v\in V$ such that 
\begin{equation*}
\label{gou3}
a(v,w)=f(w), \qquad \forall w\in W.
\end{equation*}
If, in addition to (i) and (ii), $a(v,w)$ satisfies\\
(iii) ({\bf Regularity})
for all $g\in H$, any solution $v\in V$ of \eqref{gou3} with $f(w):=(g,w)_H$ belongs to $W$, then  \eqref{gou3} has a unique solution $v\in W$.
\end{theorem}
\begin{remark}In the sense of semigroup theory, Theorem \ref{khanal} gives us the following:  Let $D(A)$ denote the set of those $v\in W$ when  $g$ ranges over $H$, and set $Av= - g$. Then $A$ is a maximal dissipative operator, and hence it generates  a continuous semigroup of contractions $(e^{tA})_{t\ge 0}$ in $H$.  
\end{remark}
\subsection{Well-posedness on $L^2_{xdx}$} This subsection is dedicated to give an answer for the well-posedness  of \eqref{ec_lin} on $L^2_{xdx}$. More precisely, for sake of simplicity, let us consider the operator $A_1u=-u_{xxxxx}-u_{xxx}$, thus, the following result can be proved.
\begin{proposition}\label{prop1}
Let $A_1u=-u_{xxxxx}-u_{xxx}$ with domain 
\[ D(A_1)=\{ u\in H^4(0,L)\cap H^2_0(0,L); \ u_{xxxxx}\in L^2_{xdx}, \ u_{xx}(L)=0\} \subset L^2_{xdx} .\]
 Then $A_1$ generates a
strongly continuous semigroup in $L^2_{xdx}$.   
\end{proposition}
\begin{proof}
Let be
\[
H=L^2_{xdx}, \quad V=H^2_0(0,L),\quad W=\{ w\in H^2_0 (0,L),\ w_{xxx}\in L^2_{x^2dx} \}, 
\]
endowed with the respective norms
\[
||u||_H : = ||\sqrt{x} u||_{L^2(0,L)},\quad ||v||_V := ||v_{xx}||_{L^2(0,L)}, \quad ||w||_W := ||xw_{xxx}||_{L^2(0,L)}. 
\]
Clearly, $V\subset H$ with a continuous (dense) embedding between two Hilbert spaces. On the other hand, we have that 
\be
\label{Q1}
||w_{xx}||_{L^2} \le  C ||xw_{xxx}||_{L^2} \qquad \forall w\in W. 
\ee  
In fact, first, we note that we have for $w\in {\mathcal T} :=C^\infty( [0,L] )\cap H^2_0(0,L)$ and $p\in \R$, the following
\begin{equation*}
\begin{split}
0\le \int_0^L (xw_{xxx}+pw_{xx}) ^2dx& = \int_0^L(x^2w^2_{xxx} + 2pxw_{xx}w_{xxx} + p^2 w_{xx}^2) dx \\&= \int_0^L x^2w_{xxx}^2 dx + (p^2-p) \int_0^L w_{xx}^2 dx + pLw_{xx}^2(L).
\end{split}
\end{equation*}
Taking $p=1/2$ results in 
\begin{equation}
\label{poi}
\int_0^L w_{xx}^2 dx \le 4\int_0^L x^2 w_{xxx}^2 dx + 2L |w_{xx}(L)|^2.
\end{equation} 
The estimate \eqref{poi} is also true for any $w\in W$, since $\mathcal T$ is dense in $W$.  Let us prove 
\eqref{Q1} by contradiction. If \eqref{Q1} is false, then there exists a sequence $\{ w^n \}_{n\ge 0}$ in $W$ such that 
\[
1 = ||w_{xx}^n||_{L^2} \ge n ||x w^n_{xxx}||_{L^2}\qquad \forall n\ge 0. 
\]
Extracting subsequences, we may assume  that
\begin{eqnarray*}
w^n&\to& w\quad \text{ in } H^2_0(0,L) \text{ weakly}\\
xw^n_{xxx} &\to& 0\quad \text{ in } L^2(0,L) \text{ strongly}
\end{eqnarray*}
and hence $xw_{xxx}=0$, which gives $w(x)=c_1x^2+c_2x+c_3$. Since $w\in H^2_0(0,L)$, we infer that $w\equiv 0 $. 
Since $w^n$ is bounded in $H^3(L/2,L)$, extracting subsequences we may also assume that $w_{xx}^n (L)$ converges in $\mathbb R$. 
We infer then from \eqref{poi} that $w^n$ is a Cauchy sequence  in $H^2_0(0,L)$, so that 
\[
w^n \to w\quad \text{ in } H^2_0(0,L) \text{ strongly},
\] 
and hence $$||w_{xx}||_{L^2} =\lim_{n\to\infty} ||w_{xx}^n||_{L^2} =1.$$ This contradicts the fact that $w\equiv 0$. The proof of \eqref{Q1} is achieved.  Thus $||\cdot ||_W$ is a norm in $W$, which is clearly a Hilbert space, and $W\subset V$ with continuous (dense) embedding. 

Define the following bilinear form on $V\times W$
\[
a(v,w) := \int_0^L v_{xx} [(xw)_{xxx} +(xw) _x] dx, \qquad v\in V, \ w\in W.
\] 
Let us check that (i), (ii), and (iii) in Theorem \ref{khanal} hold. For $v\in V$ and $w\in W$, follows that
\begin{eqnarray*}
|a(v,w)| &\le& 
||v_{xx}||_{L^2} ||xw_{xxx} + 3w_{xx} + (xw)_x ||_{L^2} \\
&\le& ||v_{xx}||_{L^2} \big( ||xw_{xxx}||_{L^2} + C||w_{xx}||_{L^2}\big) \\
&\le& C ||v||_V ||w||_W
\end{eqnarray*}
where we used Poincar\'e inequality and \eqref{Q1}. This proves that the bilinear form $a$ is well defined and continuous on $V\times W$ and, therefore (i) is proved. 

For (ii), we first pick any $w\in {\mathcal T}$ to obtain
\begin{eqnarray*}
a(w,w) &=& \int_0^L w_{xx} ( 3w_{xx} + xw_{xxx}) dx+\int_0^L w_{xx}(xw)_xdx \\
&=& \frac{5}{2} \int_0^L w_{xx}^2 dx + \left[x\frac{w_{xx}^2}{2}\right]_0^L -\frac{3}{2} \int_0^L w_x^2 dx\\
&\ge & \frac{5}{2} \int_0^L w_{xx}^2 dx - \frac{3}{2} \int_0^L w_x^2 dx.
\end{eqnarray*}
By Poincar\'e inequality
\[
\int_0^L w_x^2(x) dx \le \left( \frac{L}{\pi} \right)^2 \int_0^L w_{xx}^2(x)dx,
\]
and hence
\[
a(w,w) \ge \left(\frac{5}{2} -\frac{3L^2}{2\pi ^2}\right)\int_0^L w_{xx}^2 dx.  
\]
This shows the coercivity when $L< \pi \sqrt{\frac{5}{3}}$. When $L\ge \pi \sqrt{\frac{5}{3}}$, we have to consider, instead of $a$, the bilinear form 
$a_\lambda(v,w):=a(v,w) + \lambda (v,w)_H$ for $\lambda \gg 1$. Indeed, we have by Cauchy-Schwarz inequality and Hardy type inequality
\begin{eqnarray*}
||w||^2_{L^2} &\le & ||x^{\frac{1}{2}} w||_{L^2} ||x^{-\frac{1}{2}} w||_{L^2} \\
&\le& \sqrt{L} ||w||_{H} ||x^{-1} w||_{L^2} \\
&\le& \varepsilon ||w_{xx}||^2_{L^2} +C_\varepsilon ||w||^2_{H}
\end{eqnarray*}  
and hence, by using twice Poincaré inequality
\[
a_\lambda (w,w) \ge \left(\frac{5}{2} -\frac{\varepsilon}{2}\right) ||w||^2_V + \left(\lambda -\frac{C_\varepsilon }{2}\right) ||w||^2_{H}.
\]
Therefore, if $\varepsilon <5$ and $\lambda>C_\varepsilon/2$, then $a_\lambda$ is a continuous bilinear form which is coercive. 

To prove the regularity issue, for given $g\in H$, let us consider $v\in V$ be such that 
\[
a_\lambda(v,w) =(g,w)_H \qquad \forall w\in W,  
\]
more precisely, 
\be
\label{Q3}
\int_0^L v_{xx} ((xw)_{xxx} + (xw)_x)dx + \lambda \int_0^L v(x)w(x) x dx = \int_0^L g(x) w(x) xdx.
\ee
Picking any $w\in {\mathcal D}(0,L)$ we have
\begin{equation*}
\langle x(-v_{xxxxx}- v_{xxx} + \lambda v), w\rangle _{{\mathcal D}',{\mathcal D}}
=\langle xg, w\rangle _{{\mathcal D}',{\mathcal D}}
\qquad \forall w\in {\mathcal D} (0,L),  
\end{equation*}
and hence
\be
\label{Q2} 
-v_{xxxxx} -v_{xxx} +\lambda v = g \qquad \text{ in } {\mathcal D} '(0,L). 
\ee
Since $v\in H^2_0(0,L)$ and $g\in L^2_{xdx}$, we have that $v\in H^5(\varepsilon , L)$ for all $\varepsilon \in (0,L)$ and 
$v_{xxxxx}\in L^2_{xdx}$. Taking any $w \in {\mathcal T}$ and $\varepsilon \in (0,L)$, and scaling in \eqref{Q2} by $xw$ yields
\begin{equation*}
\begin{split}
\int_\varepsilon^L v_{xx}((xw)_{xxx} +(xw)_x)dx + \left[-v_{xxxx} (xw) -v_{xx} (xw)\right]_\varepsilon ^L\\+\left[v_{xxx} (xw)_x -v_{xx} (xw)_{xx}\right]_\varepsilon ^L = \int_\varepsilon^L (g-\lambda v)xw dx. 
\end{split}
\end{equation*}
Letting $\varepsilon \to 0$ and comparing with \eqref{Q3}, we obtain
\be\label{Q4}
\begin{split}
&Lv_{xx}(L)w_{xx}(L)=\\
&\lim_{\varepsilon \to 0} \big( \varepsilon v_{xxxx}(\varepsilon ) w(\varepsilon )+\varepsilon v_{x
x}(\varepsilon ) w(\varepsilon ) -v_{xxx}(\varepsilon )(\varepsilon w_x(\varepsilon)+w(\varepsilon))+v_{xx} (\varepsilon ) (2w_x(\varepsilon ) + \varepsilon w_{xx}(\varepsilon )) \big) .
\end{split}
\ee
Since $v_{xxxxx}\in L^2_{xdx}$, we obtain successively for some constant $C>0$ and all $\varepsilon \in (0,L)$  that
\begin{equation}\label{Q10}
\begin{split}
|v_{xxxx}(\varepsilon ) -v_{xxxx} (L) | \le \left(\int_\varepsilon ^L x  | v_{xxxxx} |^2 dx \right)^\frac{1}{2} \left(\int_\varepsilon ^L x^{-1}dx\right)^\frac{1}{2} \le C |\log \varepsilon | ,
\end{split}
\end{equation}
\begin{equation}
\begin{split} 
|v_{xx}(\varepsilon) | \le C\label{Q11}
\end{split}
\end{equation}
and
\begin{equation}
\begin{split} 
|v_{xxx}(\varepsilon) | \le C\label{Q12}.
\end{split}
\end{equation}
We infer from \eqref{Q10} that $v\in H^4(0,L)$, and hence $v\in W$. Furthermore, letting $\varepsilon \to 0$ in \eqref{Q4} and using \eqref{Q10}, \eqref{Q11} and \eqref{Q12} yields
$v_{xx}(L)=0$, since $w_{xx}(L)$ was arbitrary. We conclude that $v\in {\mathcal D} (A_1)$. Conversely, it is clear that the operator 
$A_1-\lambda $ maps ${\mathcal D} (A_1)$ into $H$, and actually onto $H$ from the above computations. 
Hence $A_1-\lambda $ generates a strongly semigroup of contractions in $H$.  
\end{proof}
\begin{remark}
Note that we can use the same approach to get the Proposition \ref{prop1} for the Kawahara operator, that is, $Au=u_{xxxxx}-u_{xxx}-u_x$. In fact, to do it just consider the following bilinear form in $V\times W$ define by 
\[
a(v,w) := \int_0^L v_{xx}(-(wx)_x+(xw)_{xxx})dx-\int_0^L v_x(xw) dx, \qquad v\in V, \ w\in W.
\] 
\end{remark}
\subsection{Well-posedness on $L^2_{(L-x)^{-1}dx}$} In this subsection we are interested to investigate the well-posedness of \eqref{ec_lin}  on $L^2_{(L-x)^{-1}dx}$. More precisely, for sake of simplicity, let us consider the operator $A_2u=-u_{xxxxx}+\beta u_{xxx}$, with $\beta\in\mathbb{R}$. Thus, the following optimal result can be proved.

\begin{proposition} \label{prop2}
Let $A_2u=-u_{xxxxx}+\beta u_{xxx}$ with domain 
\[ 
{\mathcal D}(A_2)=\{ u\in H^5(0,L)\cap H^2_0(0,L); \ u_{xxxxx}\in L^2_{ \frac{1}{L-x} dx} \text{ and }  \ u_{xx}(L)=0 \} \subset L^2_{ \frac{1}{L-x} dx} .
\]
 Then $A_2$ generates a
strongly continuous semigroup in $L^2_{\frac{1}{L-x} dx}$, for $\beta >-3/80$.   
\end{proposition}
\begin{proof}
We will use Hille-Yosida Theorem, and (partially) Theorem \ref{khanal}. Let us consider
\be
\label{AAA1}
H=L^2_{ \frac{1}{L-x} dx}, \quad V=\{u\in H^2_0(0,L), \ u_{xx}\in L^2_{\frac{1}{ (L-x)^2} dx} \}, \quad W=H^3_0(0,L),  
\ee
be endowed respectively with the norms 
\be
\label{AAA2}
||u||_H = || (L-x)^{-\frac{1}{2}} u||_{L^2}, \quad ||u||_V  = || (L-x)^{-1} u_{xx} ||_{L^2}, \quad || u ||_W = ||u_{xxx}||_{L^2}.  
\ee
Using the estimates proved in \cite[Lemma 2.1]{khanal}, we know that $V$ endowed with $||\cdot||_V$ is a Hilbert space, and that the following inequalities holds
\be
\int_0^L \frac{u^{2}}{(L-x)^{6}} \mathrm{d} x \leqslant \frac{4}{25} \int_0^L \frac{u_{x}^{2}}{(L-x)^{4}} \mathrm{d} x\qquad \forall u\in V,
\label{P1}
\ee
and 
\be
 \int_0^L \frac{u_{x}^{2}}{(L-x)^{4}} \mathrm{d} x \leqslant \frac{4}{9} \int_0^L \frac{u_{x x}^{2}}{(L-x)^{2}} \mathrm{d} x \qquad \forall u\in V.
\label{P1a}
\ee
Additionally, the following estimate is provided in \cite[Lemma 2.1]{khanal}:
\be
r^{2} \int_0^L\frac{u_{x x}^{2}}{(L-x)^{2}} \mathrm{d} x-\left(2 r+3 q r-q^{2}\right) \int_0^L \frac{u_{x}^{2}}{(L-x)^{4}} \mathrm{d} x+(1-5 q+20 r) \int_0^L \frac{u^{2}}{(L-x)^{6}} \mathrm{d} x \geqslant 0,
\label{P1aa}
\ee
for any real number $r$ and $q$. 

Using the previous inequality, we get
\be
\label{P2}
||u||_H \le  L^{\frac{5}{2}}||(L-x)^{-3}u||_H \le \frac{16}{225} L^{\frac{5}{2}}||u||_V \qquad \forall u\in V. 
\ee
Thus $V\subset H$ with continuous embedding. From Poincar\'e inequality, we have that $|| \cdot ||_W$ is a norm
on $W$ equivalent to the $H^3-$norm. On the other hand, again from Hardy type inequality
\be
\label{P2P}
\int_0^L \frac{v_{xx}^2}{(L-x)^2} dx \le C \int_0^L v_{xxx} ^2dx,
\ee   
for all $v\in H^2(0,L)$, with $v_{xx}(L)=0$. Thus, we have that 
\begin{equation}
\label{P2P2}
||v||_V \le C ||v||_W\qquad \forall v \in W,
\end{equation}
which implies $W\subset V$ with continuous embedding. It is easily  seen that ${\mathcal D} (0,L)$ is dense in $H$, $V$ and $W$. Define, for $\beta\in\mathbb{R}$, the following bilinear form on $V \times W$:
\[
a(v,w) = \int_0^L  \left[ v_{xx} \left( \frac{w}{L-x}\right)_{xxx} -\beta v_{xx}\left(\frac{w}{L-x}\right)_x\right] dx\qquad (v,w)\in V\times W. 
\] 
Then, 
\begin{eqnarray*}
|a(v,w)| 
&\le &     \left\vert \int_0^L v_{xx}\left(\frac{w_{xxx} }{L-x} + 3\frac{w_{xx}}{ (L-x)^2 } \right)dx\right\vert\\
&&+\left\vert\int_0^L v_{xx}\left(6\frac{w_x}{ (L-x)^3} +6\frac{w}{(L-x)^4} -\beta\frac{w_x}{L-x} -\beta\frac{w}{(L-x)^2}\right)dx \right\vert \\
&\le & ||w_{xxx}||_{L^2} \left\vert\left\vert\frac{v_{xx}}{L-x}\right\vert\right\vert_{L^2} + 3\left\vert\left\vert \frac{w_{xx}}{L-x}\right\vert\right\vert_{L^2} \left\vert\left\vert\frac{v_{xx}}{L-x}\right\vert\right\vert_{L^2} \\
&&+\left\vert\left\vert\frac{v_{xx}}{L-x}\right\vert\right\vert_{L^2} \left( 6  \left\vert\left\vert\frac{w_x}{ (L-x)^2}\right\vert\right\vert_{L^2} +6\left\vert\left\vert\frac{w}{ (L-x)^3} \right\vert\right\vert_{L^2} +\beta\left\vert\left\vert\frac{w}{ (L-x)^2} \right\vert\right\vert_{L^2} + ||w_x||_{L^2} \right) \\
&\le& C ||v||_V ||w||_W
\end{eqnarray*}
by \eqref{P1}, \eqref{P1a}, \eqref{P2} and \eqref{P2P2}. 
This shows that $a$ is well defined and continuous. 

Let us prove the coercivity of $a$.  To do this, we rewrite $a$ as follows:
\begin{eqnarray*}
a(v,w)&=& \int_0^L v_{xx}\left(\frac{w_{xxx} }{L-x} + 3\frac{w_{xx}}{ (L-x)^2 }  + 6\frac{w_x}{ (L-x)^3}\right) dx  \\
&& + \int_0^L v_{xx}\left(6\frac{w}{(L-x)^4} -\beta\frac{w_x}{L-x} -\beta\frac{w}{(L-x)^2}\right) dx,
\end{eqnarray*} 
for $\beta\in\mathbb{R}$ and $(v,w)\in V\times W$. Thus, for any $w\in {\mathcal D} (0,L)$, yields that
\begin{eqnarray*}
a(w,w)&=& \int_0^L w_{xx}\left(\frac{w_{xxx} }{L-x} + 3\frac{w_{xx}}{ (L-x)^2 }  + 6\frac{w_x}{ (L-x)^3}\right) dx  \\
&& + \int_0^L w_{xx}\left(6\frac{w}{(L-x)^4} -\beta\frac{w_x}{L-x} -\beta\frac{w}{(L-x)^2}\right) dx\\
&=& \frac{5}{2} \int_0^L \frac{w^2_{xx}}{(L-x)^2} dx  -15 \int_0^L \frac{w_x^2}{(L-x)^4} dx+60\int_0^L \frac{w^2}{(L-x)^6} dx\\ &&+\frac{3}{2}\beta \int_0^L \frac{w_x^2}{(L-x)^2} dx -3\beta\int_0^L \frac{w^2}{(L-x)^4} dx.
\end{eqnarray*} 
Let us split the proof of coercivity in two cases.

\vspace{0.2cm}

\noindent \textbf{Case 1:} $\beta\geq0$. 

\vspace{0.1cm}

In this case, we apply \eqref{P1} and \eqref{P1a}, to obtain
\begin{eqnarray*}
\frac{3}{2}\beta \int_0^L \frac{w_x^2}{(L-x)^2} dx -3\beta\int_0^L \frac{w^2}{(L-x)^4} dx&\ge&\frac{3}{2}\beta \int_0^L \frac{w_x^2}{(L-x)^2} dx -\frac{4}{3}\beta\int_0^L \frac{w_x^2}{(L-x)^2} dx\\&=& \frac{\beta}{6} \int_0^L \frac{w_{x}^2}{ (L-x)^2 } dx.
\end{eqnarray*} 
Thus, 
\begin{eqnarray*}
a(w,w)&\ge&\frac{\beta}{6} \int_0^L \frac{w_{x}^2}{ (L-x)^2 } dx+ \frac{5}{2} \int_0^L \frac{w^2_{xx}}{(L-x)^2} dx  -15 \int_0^L \frac{w_x^2}{(L-x)^4} dx+60\int_0^L \frac{w^2}{(L-x)^6} dx,
\end{eqnarray*} 
or equivalently,
\begin{eqnarray*}
2a(w,w)&\ge&\frac{\beta}{3} \int_0^L \frac{w_{x}^2}{ (L-x)^2 } dx+ 5\int_0^L \frac{w^2_{xx}}{(L-x)^2} dx  -30 \int_0^L \frac{w_x^2}{(L-x)^4} dx+120\int_0^L \frac{w^2}{(L-x)^6} dx.
\end{eqnarray*} 
Now, ignoring the first term and applying \eqref{P1aa}, with $r=0.4$ and $q=1$,  we get 
$$
a(w,w) \ge 0.2\int_0^L\frac{w^2_{xx}}{(L-x)^2} dx.
$$
The result is also true for any $w\in W$, by density. Showing thus that the continuous bilinear form $a(v,w)$ is coercive for $\beta\geq0$, proving thus the case 1.

\vspace{0.2cm}

\noindent \textbf{Case 2:} $\beta<0$. 

\vspace{0.1cm}

Again, applying \eqref{P1} and \eqref{P1a}, we have
\begin{eqnarray*}
\frac{3}{2}\beta \int_0^L \frac{w_x^2}{(L-x)^2} dx -3\beta\int_0^L \frac{w^2}{(L-x)^4} dx&\ge&\frac{3}{2}\beta \int_0^L \frac{w_x^2}{(L-x)^2} dx \ge6\beta \int_0^L \frac{w_x^2}{(L-x)^4} dx \\&\ge& \frac{8}{3} \beta\int_0^L \frac{w_{xx}^2}{ (L-x)^2 } dx.
\end{eqnarray*} 
Thus, for $\beta<0$,
\begin{eqnarray*}
a(w,w)&\ge&\left(\frac{5}{2}+\frac{8}{3}\beta\right)\int_0^L \frac{w_{xx}^2}{ (L-x)^2 } dx  -15 \int_0^L \frac{w_x^2}{(L-x)^4} dx+60\int_0^L \frac{w^2}{(L-x)^6} dx,
\end{eqnarray*} 
equivalently, 
\begin{eqnarray*}
2a(w,w)&\ge&\left(5+\frac{16}{3}\beta\right)\int_0^L \frac{w_{xx}^2}{ (L-x)^2 } dx  -30 \int_0^L \frac{w_x^2}{(L-x)^4} dx+120\int_0^L \frac{w^2}{(L-x)^6} dx.
\end{eqnarray*} 
Applying \eqref{P1aa} with $q$ and $r$ satisfying 
$$1-5q+20r=4(2r+3qr-q^2)>0,$$
yields that
\begin{eqnarray*}
2a(w,w)&\ge&\left(5+\frac{16}{3}\beta-\frac{30r^2}{2r+3qr-q^2}\right)\int_0^L \frac{w_{xx}^2}{ (L-x)^2 } dx.
\end{eqnarray*} 
In order for $a $ to be coercive, $\beta$ has to satisfy 
$$\beta>\frac{15}{16}\left(\frac{6r^2}{2r+3qr-q^2}-1\right).$$
Finally, considering $r=1/10$ and $q=11/20$, we have the optimal range of $\beta $, that is, $$\beta>-\frac{3}{80}.$$
Therefore, 
\be
a(w,w) \ge \left(\frac{1}{5}+\frac{16}{3}\beta\right)\int_0^L\frac{w^2_{xx}}{(L-x)^2} dx \ge \gamma || w||^2_V,
\label{AAA11}
\ee
where \begin{equation}\label{gamma}
\gamma=\frac{1}{10}+\frac{8}{3}\beta,
\end{equation}
reaching the case 2, which is also true for any $w\in W$, by density. Thus, cases 1 and 2 proves that the continuous bilinear form $a(v,w)$ is coercive for $\beta>-\frac{3}{80}.$

Now, to finish the proof, let us show that  $-A_2$ is maximal dissipative. First, consider $g\in H$ be given. By Theorem \ref{khanal}, there is at least one solution $v\in V$ of 
\be
\label{P8}
a(v,w) =(g,w)_H \qquad \forall w\in W. 
\ee  
Consider $v\in V$ a solution, let us prove that $v\in {\mathcal D} (A_2)$. Taking any $w\in {\mathcal D} (0,L)$ in \eqref{P8} yields
\be
\label{P10}
-v_{xxxxx}+\beta v_{xxx}=g \qquad \text{ in } {\mathcal D}'(0,L). 
\ee 
As $g\in L^2(0,L)$ and $v\in H^2(0,L)$, we have that $v_{xxxxx}\in L^2(0,L)$, and $v\in H^5(0,L)$. Let us take, finally, $w$ of the form 
$$w(x)=x^3(L-x)^3 \overline{w} (x),$$ where 
$\overline{w}\in C^\infty ([0,L])$ is arbitrary chosen. Note that $w\in W$ and that $$\frac{w}{(L-x)} \in H^2_0(0,L)\cap C^\infty ( [0,L] ).$$ By simplicity, consider $\beta=1$ in \eqref{P10}. Multiplying
\eqref{P10} by $w/(L-x)$ and integrating over $(0,L)$, we obtain after comparing with \eqref{P8} that
\begin{equation*}
\begin{split}
0&=v_{xx} \left.\left(\frac{w}{L-x}\right)_{xx} \right\vert _0^L\\ &= -v_{xx} \left( 6x(L-x)\overline{w}(L+x) + 2x^2(L-x)\overline{w}_x(3L-5x)+x^3(2\overline{w}+(L-x)^2\overline{w}_{xx}) \right)\vert _0^L
\end{split}
\end{equation*}
i.e.,
\[ 0=2L^3v_{xx}(L)\overline{w}(L). \]
As $\overline{w} (L)$ can be chosen arbitrarily, we conclude that $v_{xx}(L)=0$. Using \eqref{P2P} we infer that $v_{xxx}\in H$, and hence 
$v_{xxxxx}=-g +\beta v_{xxx}\in H$. 
Therefore
 $v\in {\mathcal D} (A_2)$. Thus, we have that $A_2: {\mathcal D} (A_2) \to H$ is onto.

Lastly,  let us check that $-A_2$ is also dissipative in $H$. Pick any $w\in {\mathcal D} (A_2)$. Then we obtain after some integration by parts that 
\begin{eqnarray*}
 (-A_2 w , w)_H&=&(w_{xxxxx}-\beta w_{xxx} , w)_H\\&=& -\frac{5}{2} \int_0^L \frac{w^2_{xx}}{(L-x)^2} dx  +15 \int_0^L \frac{w_x^2}{(L-x)^4} dx-\frac{3}{2}\beta \int_0^L \frac{w_x^2}{(L-x)^2} dx \\ &&-60\int_0^L \frac{w^2}{(L-x)^6} dx+3\beta\int_0^L \frac{w^2}{(L-x)^4} dx -\frac{ w_{xx}^2(0) }{2L}.
\end{eqnarray*}
 Therefore, if $\beta\geq0$, thanks to the case 1, we get
 \[
 (-A_2 w, w)_H \le - 0.1||w||_W^2 -\frac{w_{xx}^2(0)}{2L} \le 0.
 \]
 By other hand, if $\beta<0$, using the case 2, yields that
  \[
 (-A_2 w, w)_H \le - \gamma||w||_W^2 -\frac{w_{xx}^2(0)}{2L} \le 0,
 \]
 for $\gamma$ defined by \eqref{gamma}. Therefore, we conclude that $-A_2$ is maximal dissipative for $\beta>-3/80$, and thus it generates a strongly continuous semigroup of contractions in  $H$ by Hille-Yosida Theorem, achieving the proof of the proposition.
 \end{proof}
The following result, ensure a global Kato smoothing effect, as is well-know for Kawahara equation \cite{CaKawahara,khanal}.
\begin{proposition}
\label{prop40}
Let $H$ and $V$ be as in \eqref{AAA1}-\eqref{AAA2}, and let $T>0$ be given. Then there exists some constant $C=C(L,T)$ such that for any $u_0\in H$, the solution 
$u(t)=e^{tA_2}u_0$ of \eqref{ec_lin} satisfies
\begin{equation*}
||u||_{L^\infty(0,T,H)} + ||u||_{L^2(0,T,V)} \le C ||u_0||_H. 
\end{equation*}
\end{proposition}
\begin{proof}
First, we notice that ${\mathcal D}(A_2)$ is dense in $H$, so that it is sufficient to prove the result when $u_0\in {\mathcal D}(A_2)$.
Note that the estimate  $||u||_{L^\infty(0,T,H)} \le C ||u_0||_H$ is a consequence of classical semigroup theory.
 Assume $u_0\in {\mathcal D}(A_2)$, so that $u_t=-A_2u$ in the classical sense. Taking the inner product in $H$ with $u$ yields 
\[ (u_t,u)_H=-a(u,u) \le -C ||u||_V^2 \]
as done in  \eqref{AAA11}. Finally, an integration over $(0,T)$ completes the proof of the estimate of $||u||_{ L^2(0,T,V) } $. 
\end{proof}

\begin{remark}
Note that we can use the same approach to get the Propositions \ref{prop2} and \ref{prop40} for the Kawahara operator, that is, $Au=u_{xxxxx}-u_{xxx}-u_{x}$. In fact, the results follow considering the following bilinear form in $V\times W$
\begin{equation}\label{aA}
a(v,w) := \int_0^L v_{xx}\left(-\left(\frac{w}{L-x}\right)_x+\left(\frac{w}{L-x}\right)_{xxx}\right)dx-\int_0^L v_x\left(\frac{w}{L-x}\right) dx,
\end{equation}
for  $v\in V$ and $w\in W$.
\end{remark}
 
\subsection{Non-homogeneous system} We will consider in this subsection the well-posedness of the Kawahara nonhomogeneous system, namely
\begin{equation}
\left\{
\begin{array}
[c]{lll}%
u_{t}+u_{x}+u_{xxx}-u_{xxxxx}=f(x,t) &  & \text{in }(  0,T)  \times(
0,L)  \text{,}\\
u(  t,0)  =u(  t,L)  =u_{x}(  t,0)=u_{x}(  t,L)=u_{xx}(  t,L)  =0 &  &
\text{in }(  0,T)  \text{,}\\
u(  0,x)  =u_{0}(  x)  &  & \text{in }(
0,L)  \text{.}%
\end{array}
\right.  \label{H5}%
\end{equation}
More precisely, we are interested to prove the existence of a ``reasonable'' solution when $f\in L^2(0,T,H^{-2}(0,L))$. 

\begin{proposition}
\label{prop20}
Let $u_0\in L^2_{xdx}$ and $f\in L^{2}(  0,T;H^{-2}(  0,L)  )  $. Then there exists a unique solution $u\in C([0,T],L^2_{xdx} )\cap L^2(0,T,H^2(0,L))$ 
to \eqref{H5}. Furthermore, there a constant $C>0$ such that 
\be
\label{H700}
||u||_{L^\infty (0,T,L^2_{xdx})} + ||u||_{L^2(0,T,H^2(0,L))} \le C\big( ||u_0||_{L^2_{xdx}} + || f ||_{ L^2 (0,T,H^{-2}(0,L) } \big) . 
\ee
\end{proposition}
\begin{proof}
Assume first that $u_0\in {\mathcal D}(A_1)$ and $f\in C^0([0,T],{\mathcal D} (A_1))$. 
Multiplying  \eqref{H5} by 
$xu$ and integrating over $(0,\tau ) \times (0,L)$ where $0<\tau <T$ yields
\begin{equation}\label{H8} 
\begin{split}
\frac{1}{2}\int_0^L x|u(\tau ,x)|^2 dx &-\frac{1}{2}\int_0^L x|u_0(x)|^2 dx +\frac{5}{2}\int_0^\tau \!\!\!\int_0^L |u_{xx}|^2 dxdt\\ &+\frac{3}{2}\int_0^\tau \!\!\!\int_0^L |u_x|^2 dxdt
-\frac{1}{2}\int_0^\tau \!\!\!\int_0^L |u|^2 dx dt = \int_0^\tau\!\!\!\int_0^L xuf dxdt. 
\end{split}
\end{equation}
We denote $\langle .,. \rangle _{H^{-2},H^2_0}$ the duality pairing between $H^{-2}(0,L)$ and $H^2_0(0,L)$. Thus, for all $\varepsilon >0$,  we have that
\begin{equation*}
\int_0^\tau\!\!\!\int_0^L xuf  dxdt  = \int_0^\tau \langle f, xu\rangle _{ H^{-2},H^2_0 } \le \frac{\varepsilon}{2} \int_0^\tau \!\!\!\int_0^L u_x^2 dxdt + 
C_\varepsilon \int_0^\tau || f ||^2_{H^{-2}} dt. 
\end{equation*}
The last term in the left hand side of \eqref{H8} is decomposed as follows
\[
\frac{1}{2}\int_0^\tau \!\!\!\int_0^L |u|^2 dxdt = \frac{1}{2}\int_0^\tau\!\!\!\int_0^{\sqrt{\varepsilon}}  |u|^2 dxdt  + 
\frac{1}{2}\int_0^\tau \!\!\! \int_{\sqrt{\varepsilon}}^L |u|^2 dxdt =:I_1 + I_2.
\]
The following inequalities are verified
\begin{equation}\label{H11}
I_1 \le  \frac{\varepsilon}{2} \int_0^\tau \!\!\!\int_0^L |u_x|^2 dxdt
\end{equation}
and 
\begin{equation}\label{H12}
I_2 \le   \frac{1}{2\sqrt{\varepsilon} } \int_0^\tau \!\!\!\int_0^L  x |u|^2 dxdt. 
\end{equation}
Indeed, as \eqref{H12} is obvious, we prove \eqref{H11}. Note that $u(0,t)=0$, thus 
\[
|u(x,t)|\le \int_0^{\sqrt{\varepsilon}} |u_x| dx \le \varepsilon ^\frac{1}{4} \big( \int_0^{\sqrt{\varepsilon}}  |u_x|^2 dx\big) ^\frac{1}{2},
\]
 for $(t,x)\in (0,T)\times (0, \sqrt{\varepsilon} )$.  Hence 
\[
\int_0^{\sqrt{\varepsilon}} |u|^2 dx \le \varepsilon \int_0^{\sqrt{\varepsilon}} |u_x|^2 dx,
\]
which gives \eqref{H11} after integrating over $t\in (0,\tau )$.

Putting \eqref{H11} and \eqref{H12} in \eqref{H8}, we obtain that
\begin{equation*}
\begin{split}
\frac{1}{2} \int_0^L x|u(\tau , x)|^2 dx  &+\frac{5}{2}\int_0^\tau \!\!\!\int_0^L |u_{xx}|^2 dxdt + (\frac{3}{2}-\varepsilon ) \int_0^\tau \!\!\!\int_0^L |u_x|^2 dxdt \\
&\le \frac{1}{2} \int_0^L x|u_0(x)|^2 dx + \frac{1}{2\sqrt{\varepsilon}} \int _0^\tau\!\!\!\int_0^L x|u|^2 dxdt + C_\varepsilon \int_0^\tau || f ||^2_{H^{-2}}dt,
\end{split}
\end{equation*}
for $0<\varepsilon<L^2$. Taking $\varepsilon\in(L^2,\min\{0,3/2\})$ and applying Gronwall's Lemma, yields that
\[
||u||^2_{L^\infty (0,T,L^2_{xdx}) } + ||u_{xx}||^2_{L^2(0,T,L^2(0,L)) } \le 
C(T) \big( ||u_0||^2_{L^2_{xdx}} +||f||^2_{L^2(0,T,H^{-2}(0,L))} \big)  .
\]
Which proves the inequality \eqref{H700} for $u_0\in D(A_1)$ and $f\in C^0([0,T],D(A_1))$. A density argument allows us to construct a solution
$u\in C([0,T],L^2_{xdx}) \cap L^2(0,T,H^2(0,L))$ of \eqref{H5} satisfying \eqref{H700}  for $u_0\in L^2_{xdx}$ and $f\in L^2(0,T,H^{-2}(0,L))$. Finally, with respect to uniqueness, this follows from classical semigroup theory.
\end{proof}

Our aim in the next proposition is to obtain a similar result in the spaces $H$ and $V$ defined by \eqref{AAA1}-\eqref{AAA2}. To do that, we limit ourselves to the situation 
when $f=(\rho (x)h)_{xx}$ with $h\in L^2(0,T,L^2(0,L))$.  Consider $Au=u_{xxxxx}-u_{xxx}-u_x$ with domain 
\[ 
{\mathcal D}(A)=\{ u\in H^5(0,L)\cap H^2_0(0,L); \ u_{xxxxx}\in L^2_{ \frac{1}{L-x} dx} \text{ and }  \ u_{xx}(L)=0 \} \subset L^2_{ \frac{1}{L-x} dx} .
\]

\begin{proposition}
\label{prop45} 
Let $u_0\in H$, $h\in L^2(0,T,L^2(0,L))$ and set $f:=(\rho (x)h)_{xx}$.  Then there exists a unique solution $$u\in C([0,T],H)\cap L^2(0,T,V)$$ 
to \eqref{H5}. Furthermore, there is some constant $C>0$ such that
\be
\label{BBB}
||u||_{L^\infty (0,T,H)} + ||u||_{L^2(0,T,V ) } \le C\big( ||u_0||_{H} + || h ||_{ L^2 (0,T,L^2(0,L))  } \big) . 
\ee
\end{proposition}

\begin{proof}
Assume that $u_0\in {\mathcal D}(A)$ and  $h\in C_0^\infty ((0,T)\times (0,L))$, so that $f\in C^1([0,T],H)$.  
Taking the inner product of $u_t-Au-f=0$ with $u$ in $H$ yields 
\be
\label{AAA90}
 (u_t,u)_H=-a(u,u)+(f,u)_H \le  15 \int_0^L \frac{u_x^2}{(L-x)^4} dx+\frac{3}{2} \int_0^L \frac{u_x^2}{(L-x)^2} dx+(f,u)_H,
 \ee
 where $a(v,w)$ is defined by \eqref{aA}.
Then
\begin{eqnarray*}
| (f,u)_H| &=& \vert \int_0^L (\rho (x)h)_x \frac{u}{L-x} dx \vert \\
&=& \vert \int_0^L \rho (x) h \big( \frac{u_{x}}{L-x} + \frac{u}{ (L-x)^2 } \big) dx \vert \\ 
&\le&  C  ||h||_{L^2} ( || \frac{u_{x}}{L-x} ||_{L^2} + || \frac{u}{ (L-x)^2 } ||_{L^2} )  \\
&\le& C||h||_{L^2} (||u||_V+||u||_H),
\end{eqnarray*}
where we used  \eqref{P1a} in the last line. 
Thus, we have that 
\[
|(f ,u)_H |  \le \frac{C}{2} || u ||_V^2+ \frac{C}{2} || u ||_H^2 + C' || h ||_{L^2}^2. 
\]
Additionally, Hard type inequality gives
\begin{eqnarray*}
15 \int_0^L \frac{u_x^2}{(L-x)^4} dx+\frac{3}{2} \int_0^L \frac{u_x^2}{(L-x)^2} dx
&\le&  C(L)   \int_0^L \frac{u_x^2}{(L-x)^4} dx  \\
&\le&  C(L) \left(  \int_0^L \frac{u_{xx}^2}{(L-x)^4} dx +   \int_0^L \frac{u^2}{(L-x)^6} dx \right)\\
&\le& C||u||_{H}+C ||u||_V,
\end{eqnarray*}
when combined with \eqref{AAA90}, gives after integration over $(0,\tau)$ for $0<\tau<T$
\[
||u(\tau )||_{H} ^2 + C\int_0^\tau (||u||^2_V+||u||^2_H)  dt \le  ||u_0||_H^2 + C'' \big( \int_0^\tau( ||u||_H^2+||u||^2_V) dt + \int_0^\tau \!\!\!\int_0^L |h|^2 dxdt \big). 
\]
 An application of Gronwall's Lemma yields
 \eqref{BBB} for  $u_0\in {\mathcal D} (A)$ and $h\in C_0^\infty ((0,T)\times (0,L))$. 
A density argument allows us to construct a solution
$u\in C([0,T],H) \cap L^2(0,T,V)$ of \eqref{H5} satisfying \eqref{BBB}  for $u_0\in H$ and $h\in L^2(0,T,L^2(0,L))$.
The uniqueness follows from classical semigroup theory.  
\end{proof}

\section{Exact controllability for Kawahara equation\label{Sec2}}
Pick any function $\rho \in C^\infty (0,L)$ with 
\be
\label{H4}
\rho (x)= \left\{
\begin{array}{ll}
0 \quad &\textrm{ if } \ 0<x<L-\nu ,\\
1 &\textrm{ if }\  L-\frac{\nu}{2}  <x<L,
\end{array} 
\right.  
\ee
for some $\nu \in (0,L)$. This section is devoted to the investigation of the exact controllability problem for  the system%
\begin{equation}
\left\{
\begin{array}
[c]{lll}%
u_{t}+u_{x}+uu_{x}+u_{xxx}-u_{xxxxx}= f, &  & \text{in }(
0,T)  \times(  0,L)  \text{,}\\
u(  t,0)  =u(  t,L)  =u_{x}(  t,0)=u_{x}(  t,L)=u_{xx}(  t,L)  =0 &  &
\text{in }(  0,T)  \text{,}\\
u(  0,x)  =u_{0}(  x)  &  & \text{in }(
0,L)  \text{,}%
\end{array}
\right.  \label{ec1}%
\end{equation}
where $f= (\rho (x)h)_{xx} $. We aim to find a control input $h\in L^{2}(  0,T;L^{2}( 0,L)  )  $. Actually, with $(\rho (x)h(t,x))_{xx}$ in some space of functions, to 
guide the system described by (\ref{ec1}) in the time interval $[0,T]$ from any
(small) given initial state $u_{0}$ in $L^2_{\frac{1}{L-x} dx}$ to any (small) given terminal state $u_{T}$ in the same space. We first consider the linearized system,
and next proceed to the nonlinear one. To prove the main theorem we will need the results involving some weighted Sobolev spaces which was proved on the Section \ref{Sec1}.

\subsection{Exact controllability: Linearized system}
Our attention in this section is related to the control properties of the linear system
\begin{equation}\label{H21}
\left\{
\begin{array}
[c]{lll}%
u_{t}+u_{x}+u_{xxx}-u_{xxxxx}=  (\rho (x)h)_{xx}  &  & \text{in }(
0,T)  \times(  0,L)  \text{,}\\
u(  t,0)  =u(  t,L)  =u_{x}(  t,0)=u_{x}(  t,L)=u_{xx}(  t,L)  =0 &  &
\text{in }(  0,T)  \text{,}\\
u(  0,x)  =u_{0}(  x)  &  & \text{in }(
0,L)  \text{.}%
\end{array}
\right. 
\end{equation}
\begin{theorem}
\label{thm11}
Let $T>0$ , $\nu \in (0,L)$ and $\rho (x)$ as in \eqref{H4}. Then there exists a continuous linear operator
$$\Gamma : L^2_{\frac{1}{L-x} dx}\to L^2(0,T,L^2(0,L))\cap L^2_{(T-t)dt} (0,T,H^2(0,L))$$ such that for any $u_1\in L^2_{\frac{1}{L-x} dx}$, the solution 
$u$  of \eqref{H21} with  $u_0=0$ and $h=\Gamma (u_1)$ satisfies $u(T,x)=u_1(x)$ in $(0,L)$. 
\end{theorem}
\begin{proof}
We will use the Hilbert Uniqueness Method  (see e.g. \cite{Lions1}). 
Consider the following adjoint system associated to \eqref{H21}:
\begin{equation}\label{H25}
\left\{
\begin{array}
[c]{lll}%
-v_t+v_{xxxxx}-v_{xxx}-v_x= 0,&  & \text{in }(
0,T)  \times(  0,L)  \text{,}\\
v(t,0) = v(t,L) = v_x(t,0)=v_x(t,L)=v_{xx}(t,0) =0&  & \text{in }(
0,T)  \text{,}\\
v(T,x)=v_T(x). &  & \text{in }(  0,L).
\end{array}
\right. 
\end{equation}
If $u_0\equiv 0$, $v_T\in {\mathcal D} (0,L)$, and $h\in {\mathcal D} ((0,T)\times (0,L))$, multiplying \eqref{H21} by $v$ and integrating over $(0,T)\times (0,L)$, yields that
\[
\int_0^L u(T,x)v_T(x)dx =\int_0^T\!\!\!\int_0^L (\rho (x)h)_{xx} v dxdt =\int_0^T\!\!\!\int_0^L \rho (x)h v_{xx} dxdt.
\]
Considering the usual change of variables $x\to L -  x$, $t\to T-t$ and using Proposition \ref{prop20}, gives
\[
||v||_{L^\infty (0,T,L^2_{ (L-x)dx})} + || v || _{L^2(0,T,H^2 (0,L))} \le C ||v_T ||_{L^2_{ (L-x)dx} }.
\] 
By a density argument, we obtain that for all $h\in L^2(0,T,L^2(0,L))$ and all $v_T\in L^2_{(L-x)dx}$,
\[
\langle u(T,\cdot), v_T\rangle _{L^2_{ \frac{1}{L-x} dx }, L^2_{(L-x)dx}}
=\int_0^T (h,\rho (x)v_{xx})_{L^2}dt,
\]
where  $u$ and $v$  denote the solutions of \eqref{H21} and \eqref{H25}, respectively, and 
$\left\langle \cdot,\cdot\right\rangle _{L^2_{ \frac{1}{L-x} dx}, L^2_{(L-x)dx}}$ denotes the duality pairing between
$L^2_{\frac{1}{L-x} dx}$ and $L^2_{( L-x)  dx}$. We have to prove the following observability inequality
\begin{equation}
\label{observabilite}
||v_T||^2_{L^2_{ (L-x)dx}} \le C \int_0^T\!\!\!\int_0^L |\rho (x)v_{xx}|^2dxdt
\end{equation}
or, equivalently, letting $w(t,x)=v(T-t,L-x)$, 
\be
\label{H30}
||w_0||^2 _{L^2_{xdx}} \le C \int_0^T\!\!\!\int_0^L |\rho (L-x)w_{xx}|^2 dxdt,
\ee
where $w$ solves
\be
\label{H40}
\left\{ 
\begin{array}{l}
w_t-w_{xxxxx}+w_{xxx}+w_x=0,\\
w(t,0)=w(t,L)=w_x(t,0)=w_x(t,L)=w_{xx}(t,L)=0,\\
w(0,x)=w_0(x).
\end{array}
\right.
\ee
Multiplying \eqref{H40} by $wq$, for $q(t,x)=(T-t)b(x)\in C^\infty ([0,T]\times [0,L])$  where $b\in C^\infty ([0,L]) $ is nondecreasing defined by 
\[
b(x) = \left\{ 
\begin{array}{ll}
x & \text{ if }\  0<x<\nu /4 ,\\
1 & \text{ if }\  \nu /2 <x<L,
\end{array}
\right.
\]
with $\nu\in(0,L)$, after integrating by parts we have
\begin{equation*}
\begin{split}
-\int_0^T\!\!\! \int_0^L (q_t+q_x&+q_{xxx}-q_{xxxxx})\frac{w^2}{2} dxdt + \int_0^L (q\frac{w^2}{2})(T,x)dx- \int_0^L (q\frac{w^2}{2})(0,x)dx \\&+\frac{3}{2}\int_0^T\!\!\!\int_0^L q_xw_x^2 dxdt 
+\frac{5}{2}\int_0^T\!\!\!\int_0^L q_x w_{xx}^2 dxdt+\int_0^T\left(q\frac{w_{xx}^2}{2}\right)(t,0)dt =0.
\end{split}
\end{equation*}
Due the choose of $q(t,x)$ and $b(x)$, this yields
\begin{equation}\label{H45}
\begin{split}
||w_0||^2_{L^2_{xdx}}  
&\le C(L,\nu ) \int_0^L b(x) w_0^2 (x) dx  \\
&\le C(T,L,\nu ) \left(  \int_0^T\!\!\! \int_0^{\frac{\nu}{2}}  w_{x}^2 dxdt+ \int_0^T\!\!\! \int_0^{\frac{\nu}{2}}  w_{xx}^2 dxdt + \int_0^T\!\!\!\int_0^L w^2 dxdt\right) \\
&\le C(T,L,\nu ) \left(  \int_0^T\!\!\! \int_0^{\frac{\nu}{2}}  w_{xx}^2 dxdt + \int_0^T\!\!\!\int_0^L w^2 dxdt\right) ,
\end{split}
\end{equation}
using Poincaré inequality. We claim that
\be
\label{H48}
||w_0||^2_{L^2_{xdx}} \le C \int_0^T\!\!\!\int_0^{\frac{\nu}{2}} w_{xx}^2 dxdt,
\ee
holds. In fact, if the estimate \eqref{H48} does not occurs, then one can find a sequence $\{w_0^n\}\subset L^2_{xdx}$ such that 
\be
\label{H50}
1=||w_0^n||^2_{L^2_{xdx}} >n \int_0^T\!\!\!\int_0^{\frac{\nu}{2}} |w_{xx}^n|^2 dxdt,
\ee
where $w^n$ denotes the solution of \eqref{H40} with $w_0$ replaced by $w_0^n$.  By \eqref{H700} and \eqref{H50}, $\{ w^n\} $ is bounded 
in $L^2(0,T,H^2(0,L))$, hence also in $H^1(0,T,H^{-3}(0,L))$ thanks the equation \eqref{H40}. Extracting a subsequence if necessary, Aubin-Lions' Lemma  ensures that
$$w^n\to w \qquad \text{ in } L^2(0,T,H^2(0,L)).$$ Thus, using \eqref{H45} and \eqref{H50}, we see that $w_0^n$ is a Cauchy sequence in $L^2_{xdx}$,
and hence it converges strongly in this space. Let $w_0$ denote its limit in $L^2_{xdx}$, and let $w$ denote the corresponding solution of \eqref{H40}.
Then 
$$||w_0||_{L^2_{xdx}} = 1$$
and 
$$w^n \to w \qquad \text{ in } L^2(0,T,H^2(0,L)).$$
But $w_{xx}^n\to 0$ in $L^2(0,T,L^2(0,\nu /2))$ by \eqref{H50}. Thus $w_{xx}\equiv 0$ in $(0,T)\times (0,\nu/2 )$, and hence 
$w(t,x) = xg(t)+c(t)$ (for some functions $g$ and $c$) in $(0,T)\times (0,\nu /2 )$. Since $w$ satisfies \eqref{H40}, we infer from $w(t,0)=w_x(t,0)=0$ that $w\equiv 0$ in $(0,T)\times (0,\nu/2 )$. By Holmgren's theorem we have that  $w\equiv 0$
in $(0,T)\times (0,L)$, implying that $w(0,x)=0$, which is a contradiction with $|| w_0||_{L^2_{xdx}}=1$.  Therefore \eqref{H48} is proved, 
and \eqref{H30} follows.

Let us now apply the Hilbert Uniqueness Method. Consider the following operator 
$$\Lambda :  L^2_{ (L-x) dx} \to L^2_{ (L-x) dx}$$
defined by $$\Lambda (v_T) = (L-x)^{-1} u(T,\cdot)\in L^2_{ (L-x) dx},$$ where $u$ solves \eqref{H21} with 
$h=\rho (x) v_{xx}$. Then operator $\Lambda$ is clearly continuous. On the other hand, from
\eqref{observabilite}
\[
\big( \Lambda (v_T), v_T\big)_{L^2_{ (L-x)dx }} = \langle u(T,\cdot), v_T\rangle _{L^2_{ \frac{1}{L-x} dx }, L^2_{(L-x)dx}} 
=\int_0^T ||\rho (x)v_{xx}||^2_{L^2} dt \ge C ||v_T||^2_{L^2_{(L-x) dx}} ,
\]
and it follows that the map $v_T \to \Lambda (v_T)$ is invertible in $L^2_{ (L-x)dx } $. 

Define the map $$\Gamma: \ L^2_{\frac{1}{L-x} dx} \to L^2(0,T,L^2(0,L))$$ by
$\Gamma (u) = h := \rho (x)v_{xx}$, 
where $v$ is the solution of \eqref{H25} with $v_T=\Lambda ^{-1}( (L-x)^{-1}u(T,\cdot)$. 

Firstly, $\Gamma$ is continuous, and the solution $u$ of \eqref{H21} with $u_0=0$ and $h=\Gamma (u)$ satisfies $u(T,\cdot)=u_1$. 
To prove that $\Gamma $  is also continuous from  $L^2_{\frac{1}{L-x}dx}$  into  
$L^2_{ (T - t ) dt} (0,T,H^2(0,L))$, it is sufficient to show the following
estimate 
\[
\int_0^T ||v (t) ||^2_{H^3}  (T-t) dt \le C ||v_T||^2_{L^2_{(L-x)dx}},
\]
for the solutions of \eqref{H25} or, equivalently,
\be
\label{H60}
\int_0^T ||w||^2_{H^3}\, tdt \le C ||w_0||^2_{L^2_{xdx}},
\ee
for the solutions of \eqref{H40}.  
Thanks to Proposition \ref{prop20},  we have
\be
\label{H61}
\int_0^T ||w||^2_{H^2_0(0,L)} dt \le C ||w_0||^2_{L^2_{xdx}},
\ee
which yields, for $w_0\in L^2(0,L)$, that
$$
\int_0^T ||w||^2_{H^2_0(0,L)} dt \le C ||w_0||^2_{L^2}.
$$
Assume now that $w_0\in {\mathcal D}(A)$ and let $u_0=Aw_0=w_{0,xxxxx}-w_{0,xxx}-w_{0,x}$.
Denote by $w$ (resp. $u$) the solution of \eqref{H40} with initial data $w_0$ (resp. $u_0$). Then
\[
Aw=w_{xxxxx}-w_{xxx}-w_x=u\in L^2(0,T,H^2_0(0,L)), 
\]
and we infer that $w\in L^2(0,T,H^7(0,L))$. By interpolation, this gives that $$w\in L^2(0,T,H^3(0,L))$$ if $w_0\in H^2_0(0,L)$, with an estimate of the form
\be
\label{H63}
\int_0^T ||w||^2_{H^3(0,L)} dt \le C ||w_0||^2_{ H^2_0(0,L) }.
\ee
The different constants $C$ in \eqref{H61}-\eqref{H63} may be taken independent of $T$ for $0<T<T_0$. Thus, finally, due to Fubini's Theorem we get 
\[
\int_0^T s||w(s)||^2_{H^3} ds = \int_0^T\!\!\!  \left( \int_t^T ||w(s)||^2_{H^3}ds\right) dt \le C \int_0^T ||w (t) ||^2_{H^2_0(0,L)} dt \le C ||w_0||^2_{L^2_{xdx}}. 
\]  
This completes the proof of \eqref{H60} and, consequently, Theorem \ref{thm11} is shown.
\end{proof}

\begin{remark}
It is important to note that the forcing term $f=(\rho (x)h)_{xx}\in L^2_{(T-t) dt}(0,T,L^2(0,L))$ is in fact supported in $(0,T)\times (L-\nu ,L)$.  
\end{remark}

\subsection{Exact controllability: Nonlinear system}Let us prove the local exact controllability in $L^2_{\frac{1}{L-x} dx}$ of the system  (\ref{ec1}). Note that the solutions of (\ref{ec1}) can be written as 
\[
u=u_{L}+u_{1}+u_{2}\text{,}%
\]
where $u_{L}$ is the solution of (\ref{ec_lin}) with
initial data $u_{0}  \in L^2_{ \frac{1}{L-x}  dx}$, $u_{1}$ is solution of
\begin{equation}
\left\{
\begin{array}
[c]{lll}%
u_{1,t}+u_{1,x}+u_{1,xxx}-u_{1,xxxxx}= f = (\rho (x)h)_{xx}  &  & \text{in }( 0,T)  \times(  0,L)  \text{,}\\
u_{1}(  t,0)  =u_{1}(  t,L)  =u_{1,x}(  t,0)=u_{1,x}(  t,L) =u_{1,xx}(  t,L)=0 &  & \text{in }(  0,T)  \text{,}\\
u_{1}(  0,x)  =0 &  & \text{in }(  0,L)
\end{array}
\right.  \label{non_1}%
\end{equation}
with $h=h(  t,x)  \in L^{2}(  0,T;L^2( 0,L)  ) $, 
and $u_{2}$ is solution of
\begin{equation}
\left\{
\begin{array}
[c]{lll}%
u_{2,t}+u_{2,x}+u_{2,xxx}-u_{2,xxxxx}=g(  t,x)  &  & \text{in }(
0,T)  \times(  0,L)  \text{,}\\
u_{2}(  t,0)  =u_{2}(  t,L)  =u_{2,x}(  t,0)=u_{2,x}(  t,L)=u_{2,xx}(  t,L)
=0 &  & \text{in }(  0,T)  \text{,}\\
u_{2}(  0,x)  =0 &  & \text{in }(  0,L)  \text{,}%
\end{array}
\right.    \label{non_2}
\end{equation}
with $g=g(t,x)  = - uu_{x}$.

The following result is concerned with the solutions of the non-homogeneous system \eqref{non_2}.

\begin{proposition}
\label{prop_weight} 
Consider $H$ and $V$ defined as in \eqref{AAA1}-\eqref{AAA2}.
\begin{itemize}
\item[(i)] If $u,v\in L^{2}(  0,T;V )$, then $uv_{x}\in L^{1}(  0,T;H)  $. Furthermore, the map%
\[
(u,v)\in L^{2}(  0,T;V)^2  \to uv_{x}\in
L^{1}(  0,T;H)
\]
is continuous and there exists a constant $c>0$ such that
\begin{equation}
\left\Vert uv_{x}\right\Vert _{L^{1}(  0,T;H)  }\leq c\left\Vert u\right\Vert  _{L^{2}( 0,T;V)  }  \left\Vert v\right\Vert  _{L^{2}( 0,T;V)  }  
\text{.} \label{non_3}
\end{equation}
\item[(ii)]  For $g\in L^{1}(  0,T;H)  $, the mild solution $u$ of \eqref{non_2} given by Duhamel formula satisfies
\[
u_{2}\in C(  \left[  0,T\right]  ;H)  \cap L^{2}(  0,T;V)
=: \mathcal{G}
\]
and we have the estimate 
\be
\label{CCC}
||u_2 ||_{L^\infty (0,T,H) } + ||u _2||_{L^2(0,T,V)} \le C ||g||_{L^1(0,T,H)}.
\ee
\end{itemize}
\end{proposition}

\begin{proof}
For $u,v\in V$, we have 
\[
||uv_x||_{L^2_{ \frac{1}{L-x}dx }} \le || u ||_{L^\infty} || \frac{v_x}{\sqrt{L-x} }||_{L^2}\le C ||u||_V ||v||_V,
\]
and (i) holds. For (ii), we first assume that $g\in C^1([0,T],H)$, so that $u_2\in C^1([0,T],H)\cap C^0([0,T],{\mathcal D} (A_2))$. Taking the inner product
of $u_{2,t}=A_2u_2+g$ with $u_2$  in $H$ yields
\begin{equation*}
(u_{2,t},u_2)_H \le -C||u_2||^2_V + C' ||u_2||_H^2 + (g,u_2)_H 
\end{equation*}
where $C,C'$ denote some positive constants.
Integrating over $(0,T)$ and using the classical estimate 
\[ 
||u_2||_{L^\infty (0,T,H) } \le C ||g||_{L^1(0,T,H)}  
\]  
coming from semigroup theory, we obtain  (ii) when $g\in C^1([0,T],H)$. The general case ($g\in L^1(0,T,H)$) follows by density.  \end{proof}

Let $\Theta _1 (h):=u_1$ and $\Theta _2(g):=u_2$, where $u_1$ (resp. $u_2$) denotes the solution of \eqref{non_1} (resp. \eqref{non_2}). 
Then  $$\Theta_{1}:L^{2}(  0,T;L^2(  0,L)  )  \to \mathcal{G}$$ and $$\Theta_{2}:L^{1}(  0,T;L^2_{ \frac{1}{L-x}  dx}      )\to \mathcal{G}$$  are
well-defined continuous operators, due the Propositions \ref{prop45} and \ref{prop_weight}. 

Using Proposition \ref{prop_weight} and the contraction mapping principle, one can prove as in \cite{CaKawahara,GG,khanal} 
the existence and uniqueness of
a solution  $u\in {\mathcal G} $ of \eqref{ec1} when the initial data $u_0$ and the forcing term $h$ are small enough. As the proof is similar to those of Theorem 
\ref{exact_nlin}, we will omit it.  

We are in position to prove the main result of Section \ref{Sec3}, namely the (local) exact controllability of system \eqref{ec1}.

\begin{theorem}
\label{exact_nlin} Let $T>0$. Then there exists $\delta>0$ such that for any
$u_{0}$, $u_{1}\in L_{  \frac{1}{L-x}  dx}^{2}$ satisfying
$$\left\Vert u_{0}\right\Vert _{L^2_{  \frac{1}{L-x}  dx}} \leq \delta \quad\text{ and }	\quad    \left\Vert u_{1}\right\Vert _{L^2_{  \frac{1}{L-x}  dx}}\leq\delta,$$ 
one can find a control function $h\in L^2(  0,T;L^2 (  0,L)  )  $ such that the solution $u\in {\mathcal G} $ of (\ref{ec1})
satisfies $u(  T,\cdot)  =u_1$ in $(  0,L)  $.
\end{theorem}
\begin{proof}
To show the result, we will apply the contraction mapping principle. Let $\mathcal{F}$ denote
the nonlinear map%
\[
\mathcal{F}:L^{2}(  0,T; V )  \to \mathcal{G},%
\]
defined by%
\[
\mathcal{F}(  u   )  =u_{L} +\Theta_1 \circ\Gamma(  u_{T}-u_{L}( T,\cdot)  +\Theta_ 2(  uu_{x})  (  T,\cdot) )  - \Theta_2( uu_{x})  \text{,}%
\]
where $u_{L}$ is the  solution of (\ref{ec_lin}) with initial data $u_{0}  \in L^2_{\frac{1}{L-x}dx}$, 
$\Theta_{1}$ and $\Theta_{2}$ are defined as above and $\Gamma$ is defined in Theorem \ref{thm11}.

Observe that if $u$ is a fixed point of $\mathcal{F}$, then $u$ is a solution
of (\ref{ec1}) with the control $$h=  \Gamma(  u_{T}-u_{L}( T,\cdot)  +\Theta_ 2(  uu_{x})  (  T,\cdot)),$$ 
and satisfies
\[
u(  T,\cdot)  =u_{T},%
\]
as desired.
In order to prove the existence of a fixed point of $\mathcal{F}$, we apply the Banach
fixed-point Theorem to the restriction of $\mathcal{F}$ to some closed ball
$\overline{B(0,R)}$ in $L^{2}(  0,T;V) $.

\vspace{0.2cm}

\noindent (i) $\mathcal{F}$ \textit{is contractive. } 

\vspace{0.2cm} 

Pick any $u,\tilde u\in \overline{B(0,R)}$. 
Using  \eqref{BBB}, \eqref{non_3} and \eqref{CCC}, we have
\begin{equation}
\left\Vert \mathcal{F}(  u)  -\mathcal{F}(  \tilde{u})
\right\Vert _{L^{2}(  0,T; V)  }%
\leq2CR\left\Vert u-\tilde{u}\right\Vert _{L^{2}(  0,T; V )  }\text{,} \label{fixed1}%
\end{equation}
for some constant $C>0$, independent of $u$, $\tilde{u}$ and $R$. Hence, $\mathcal{F}$ is contractive if $R$ satisfies
\begin{equation}
R<\frac{1}{4C}\text{,} \label{fixed2}%
\end{equation}
where $C$ is the constant in \eqref{fixed1}.

\vspace{0.2cm}

\noindent\noindent  (ii) $\mathcal{F}$ \textit{maps }$\overline{B(0,R)}$
\textit{into itself.}  

\vspace{0.2cm}

Using Proposition \ref{prop40} and the continuity of the operators $\Gamma$, $\Theta_1$ and $\Theta_2$, we infer the existence of a constant $C'>0$ such that 
for any $u\in\overline{B(0,R)}$, we have
\[
\left\Vert \mathcal{F}(  u)  \right\Vert _{L^{2}( 0,T;V)  }\leq C'(  \left\Vert u_{0}%
\right\Vert _{L_{  \frac{1}{L-x}  dx}^{2}}+\left\Vert
u_{T}\right\Vert _{L_{  \frac{1}{L-x}  dx}^{2}}+R^{2})
\text{.}%
\]
Thus, taking $R$ satisfying \eqref{fixed2}, $$R<1/(2C')$$ and assuming that $\left\Vert u_{0}\right\Vert
_{L_{ \frac{1}{L-x}  dx}^{2}}$ and $\left\Vert u_{T}\right\Vert
_{L_{  \frac{1}{L-x}  dx}^{2}}$ are small enough, we obtain that the operator
$\mathcal{F}$ maps $\overline{B(0,R)}$ into itself. Therefore the map
$\mathcal{F}$ has a fixed point in $\overline{B(0,R)}$ by the Banach fixed-point Theorem. The proof 
of Theorem \ref{exact_nlin} is complete.
\end{proof}
\begin{remark}
As in the linear case, the forcing term $f=(\rho (x)h)_{xx}$ indeed is a function in $$L^2_{ (T-t ) dt}(0,T,L^2(0,L))$$ supported in $(0,T)\times (L-\nu ,L)$.
\end{remark}

\section{Regional controllability for Kawahara equation\label{Sec3}}
In this section we prove a regional controllability of the following system
\begin{equation}
\left\{
\begin{array}
[c]{lll}%
u_{t}+u_{x}+uu_{x}+u_{xxx}-u_{xxxxx}= f  &  & \text{in }(  0,T)  \times(  0,L)  \text{,}\\
u(  t,0)  =u(  t,L)  =u_{x}(  t,0)=u_{x}(  t,L)  =u_{xx}(  t,L)=0 &  & \text{in }(  0,T)  \text{,}\\
u(  0,x)  =u_{0}(  x)  &  & \text{in }( 0,L)  \text{.}%
\end{array}
\right.  \label{S1a}%
\end{equation}
In detail, we prove that internal control of the Kawahara equation gives a control of hyperbolic type in the left direction and a control of parabolic type in the right direction. Before presenting the proof of the result we remark that the existence of a solution for the system \eqref{S1a} in the Sobolev space was shown in \cite{GGa} (see also \cite{MoChen}).

Now, let us state and prove the main result of this section.
\begin{theorem}
\label{thmC1}
Let $T>0$ and $\omega = (l_1,l_2)$ with $0<l_1<l_2<L$.  Pick any number $l_1'\in (l_1,l_2)$. Then there exists a number $\delta >0$ such that  for any
$u_0,u_1\in L^2(0,L)$ satisfying 
$$\left\Vert u_{0}\right\Vert _{L^2(0,L)} \leq \delta \quad\text{ and }	\quad    \left\Vert u_{1}\right\Vert _{L^2(0,L)}\leq\delta,$$ 
one can find a control $f\in L^2(0,T,H^{-2}(0,L))$ with $\text{supp}(f)\subset (0,T)\times \omega$ such that the solution 
$u\in C^0([0,T],L^2(0,L)) \cap L^2(0,T,H^2 (0,L))$ of \eqref{S1a} satisfies 
\begin{equation}\label{corS1}
u(T,x) =\left\{ 
\begin{array}{ll}
u_1(x) & \text{ if }  \ x\in (0,l_1');\\
0 &\text{ if } \ x\in (l_2,L).
\end{array}
\right. 
\end{equation}
\end{theorem}
\begin{proof}
By \cite[Theorem 1.1]{MoChen}, if $\delta$ is small enough one can find a control input $f\in L^2(0,T/2,L^2(0,L))$ with 
$\text{supp} (f) \subset (0,T)\times \omega$ such that the solution of 
\eqref{S1a} satisfies $u(T/2,.)\equiv 0$ in $(0,L)$, where $\omega$ is a subset of $(0,L)$.

Let us consider any number $l_2'\in (l_1',l_2)\subset(0,L)$. By \cite[Theorem 1]{GG}, if $\delta$ is small enough one can pick a function $g,h\in L^2(T/2,T)$ such that the solution $$y\in C^0([T/2,T],L^2(0,l_2'))\cap L^2(T/2,T,H^2(0,l_2'))$$ of the system
\[
\left\{
\begin{array}{ll}
y_t-y_{xxxxx}+y_{xxx}+y_x +yy_x =0 \quad &\text{ in } (T/2,T)\times (0,l_2'),\\
y(t,0)=y_x(t,0)=y_{xx}(t,l_2')=0,\ \ y(t,l_2')=g(t),\ \ y_x(t,l_2')=h(t) &\text{ in } (T/2,T),\\ 
y(T/2,x)=0 &\text{ in } (0,l_2')
\end{array}
\right.
\]
satisfies $y(T,x)=u_1(x)$ for $0<x<l_2'$. Define a function $\mu\in C^\infty ( [0,L] ) $ as
\[
\mu (x) :=\left\{ 
\begin{array}{ll}
1\quad &\text{ if } x<l_1',\\
0&\text{ if } x>\frac{l_1'+l_2'}{2},
\end{array}
\right. 
\]
and set for $T/2< t \le T$
\[
u(t,x)=
\left\{ 
\begin{array}{ll}
\mu(x) y(t,x)\quad &\text{ if } x<l_2',\\
0&\text{ if } x>l_2'.
\end{array}
\right.  
\]
Note that, for $T/2<t<T$, $u_t-u_{xxxxx}+u_{xxx}+u_x+uu_x=f$ with 
\begin{eqnarray*}
f&=& -(\mu'''''y+5\mu''''y_x+10\mu'''y_{xx}+10\mu''y_{xxx}+5\mu'y_{xxxx}) \\
&&+(\mu'''y+3\mu''y_x+3\mu'y_{xx}+\mu'y)+	\mu\mu'y^2+\mu(\mu-1)yy_x.
\end{eqnarray*}
Since $||y||^4_{ L^4(0,T,L^4(0,l_2'))}\le C||y||^2_{L^\infty(0,T,L^2(0,L))} ||y||^2_{L^2(0,T,H^2(0,L))}$, it is clear that $$f\in L^2(0,T,H^{-2}(0,L))$$ with 
$\text{supp} (f)\subset (0,T)\times (l_1,l_2)$. Furthermore, 
$u\in C([0,T],L^2(0,L))\cap L^2(0,T,H^2(0,L))$ solves \eqref{S1a} and satisfies  \eqref{corS1}, proving the result.
\end{proof}

\section{Further Comments and Open issues\label{Sec4}}

In this work we treated the well-posedness and controllability of the Kawahara equation, a fifth order KdV type equation, in a bounded domain. Here, we were able to give an almost complete picture of the internal controllability for the Kawahara system started by \cite{MoChen}. Thus, the following remarks are now in order.

\begin{itemize} 
\item[i.] A result of the controllability to the trajectories remains valid for the system \eqref{S1}. Precisely,  the result can be read as follows
\begin{theorem}\label{trajectories} 
Let $\omega=\left(l_{1}, l_{2}\right)$ with $0<l_{1}<l_{2}<L,$ and let $T>0 .$ For $\bar{u}_{0} \in L^{2}(0, L),$ let $\bar{u} \in$ $C^{0}\left([0, T] ; L^{2}(0, L)\right) \cap L^{2}\left(0, T ; H^{2}(0, L)\right)$ denote the solution of 
\begin{equation}\label{S1A}
 \left\{\begin{array}{ll} \bar{u}_{t}+\bar{u}_{x}+\bar{u} \bar{u}_{x}+\bar{u}_{x x x}-\bar{u}_{x x x x x}=0, & \text { in }(0, T) \times(0, L), \\ \bar{u}(t, 0)=\bar{u}(t, L)=\bar{u}_{x}(t, 0)=\bar{u}_{x}(t, L)=\bar{u}_{xx}(t, L)=0, & \text { in }(0, T), \\ \bar{u}(0, x)=\bar{u}_{0}(x), & \text { in }(0, L).
 \end{array}\right. 
\end{equation} 
Then, there exists $\delta>0$ such that for any $u_{0} \in L^{2}(0, L)$ satisfying $\left\|u_{0}-\bar{u}_{0}\right\|_{L^{2}(0, L)} \leq \delta,$ there exists $f \in$ $L^{2}((0, T) \times \omega)$ such that the solution $$u \in C^{0}\left([0, T] ; L^{2}(0, L)\right) \cap L^{2}\left(0, T, H^{2}(0, L)\right)$$ of \eqref{S1} satisfies $u(T, \cdot)=\bar{u}(T, \cdot)$ in $(0, L)$. \end{theorem} 

\item[ii.] The proof of Theorem \ref{trajectories} is a direct consequence of the Carleman estimate shown by Chen \cite{MoChen}, being precise: \cite[Theorem 1.1]{MoChen} is equivalent to the previous result. In fact, consider $u$ and $\bar{u}$ fulfilling the system \eqref{S1} and  \eqref{S1A}, respectively. Then $q=u-\bar{u}$ satisfies
\begin{equation}\label{S1AA}
\left\{\begin{array}{ll}
q_{t}+q_{x}+\left(\frac{q^{2}}{2}+\bar{u} q\right)_{x}+q_{x x x}-q_{x x x x x}=1_{\omega} f(t, x), & \text { in }(0, T) \times(0, L), \\
q(t, 0)=q(t, L)=q_{x}(t,0)=q_{x}(t, L)=q_{xx}(t, L)=0, & \text { in }(0, T), \\
q(0, x)=q_{0}(x):=u_{0}(x)-\bar{u}_{0}(x), & \text { in }(0, L).
\end{array}\right.
\end{equation}
So, the objective is to find $f$ such that the solution $q$ of \eqref{S1AA} satisfies
$$q(T, \cdot)=0.$$
However, this is exactly what has been proven in \cite[Theorem 1.1]{MoChen}, this means that the null controllability for the Kawahara equation is equivalent to the controllability to the trajectories for this equation.
\end{itemize} 

Observe that with the Theorems \ref{thmB}, \ref{thmC}, \ref{trajectories} and \cite[Theorem 1.1]{MoChen} we have almost completed the answers regarding internal controllability for equation \eqref{S1}. However,  it is important to note that due to the techniques used here the issue whether $u$ may also be controlled in the interval $(l_1',l_2)\subset(0,L)$ is open, missing a final step to give a complete answer on Kawahara's internal controllability. This open problem can be presented as follows:

\vspace{0.2cm}
\noindent\textbf{Problem $\mathcal{A}$}: Is it possible to control the Kawahara equation in the interval $(l_1',l_2)$?
\vspace{0.1cm}

Anyway, other problems about internal controllability can be attacked using new techniques and arguments. In this way, below, our plan is to present some problems that seem interesting from a mathematical point of view. More precisely, we present open issues about internal controllability of the Kawahara equation with an integral condition in unbounded and bounded domains.

\subsection{Controllability of Kawahara equation: Unbounded domain} In the context of control on unbounded domains, Faminskii \cite{Fa}, in a recent work, considered the initial-boundary value problems, posed on infinite domains for the Korteweg--de Vries equation. Precisely, he elected a function $f_0$ on the right-hand side of the equation as an unknown function, regarded as a control. Thus,  the author proved that this function must be chosen such that the corresponding solution should satisfy certain additional integral conditions.

Thus, we believe that this techniques can be applied  for the Kawahara equation posed on the right/left half-lines:
\begin{equation}\label{f2}
\begin{cases}
u_{t}+u_{x}+u_{xxx}-u_{xxxxx}+ uu_{x}=f_0(t)v(x,t), & (t,x)\in (0,T)  \times(0,\infty),\\
u(0,x)=u_0(x),                                   & x\in(0,\infty),\\
u(t,0)=h(t),\ u_x(t,0)=g(t), & t\in(0,T),
\end{cases}
\end{equation}
and 
\begin{equation}\label{f3}
\begin{cases}
u_{t}+u_{x}+u_{xxx}-u_{xxxxx}+ uu_{x}=f_0(t)v(x,t), & (t,x)\in (0,T)  \times(-\infty,0),\\
u(0,x)=u_0(x),                                   & x\in(-\infty,0),\\
u(t,0)=h(t),\ u_x(t,0)=g(t),\ u_{xx}(t,0)=k(t) & t\in(0,T).
\end{cases}
\end{equation}
Here $v$ is a given function and $f_0$ is an unknown control function. Therefore, the following open issue naturally appears.

\vspace{0.2cm}
\noindent\textbf{Problem $\mathcal{B}$}: Can we find a pair $\{f_0, u\}$, satisfying $$\int_{\mathbb{R^+}}u(t,x)w(x)dx=\varphi(t), \quad \text{or} \quad \int_{\mathbb{R^-}}u(t,x)w(x)dx=\varphi(t), $$ such that the functions $w$ and $\varphi$ are given and $u$ is the solution of \eqref{f2} or \eqref{f3}?

\subsection{Controllability of Kawahara equation: Bounded domainn}
With respect to the control issues in a bounded domain a new approach, different from the one used in this article, was recently introduced by Faminskii \cite{Fa1}.  Faminskii established results for the Korteweg--de Vries equation in a bounded domain under an integral overdetermination condition. More precisely, with smallness conditions on either the input data or the time interval, the author showed the controllability when the control has a special form.

In this spirit, we believe that the following problem seems very interesting. Consider the Kawahara equation as follows: 
\begin{equation}\label{f4}
\begin{cases}
u_{t}+u_{x}+u_{xxx}-u_{xxxxx}+ uu_{x}=f_0(t)v(x,t), & (t,x)\in (0,T)  \times(0,L),\\
u(0,x)=u_0(x),                                   & x\in(0,L),\\
u(  t,0)  =h_1(t),\ u(  t,L)  =h_2(t),&t\in(0,T),\\
	 u_{x}(  t,0)=h_3(t),\ u_{x}(  t,L)=h_4(t),&t\in(0,T),\\
	 u_{xx}(  t,L)  =h_5(t)& t\in(0,T).
\end{cases}
\end{equation}

\vspace{0.2cm}
\noindent\textbf{Problem $\mathcal{C}$}: For given functions $u_0$ and $h_i$, $i=1,2,3,4,5$, can we find a function $f_0$ such that the solution $u$
of system \eqref{f4} satisfies the overdetermination condition $$\int^L_0u(t,x)w(x)dx=\varphi(x), \quad t\in(0,T)$$ where $w$ and $\varphi$ are known functions?

\subsection*{Acknowledgments:} The authors thank the anonymous referee for their helpful comments and suggestions. 

R. de A. Capistrano--Filho was supported by CNPq 306475/2017-0, 408181/2018-4, CAPES-PRINT 88881.311964/2018-01, MATHAMSUD  88881.520205/2020-01 and Propesqi (UFPE) \textit{via} ``produ\c{c}\~{a}o qualificada". M. Gomes was partially supported by CNPq. This work is part of the PhD thesis of M. Gomes at Universidade Federal de Pernambuco.




\begin{thebibliography}{99}                                                                                



\bibitem {Berloff}N. Berloff, L. Howard, \textit{Solitary and periodic
solutions of nonlinear nonintegrable equations}, Studies in Applied
Mathematics, \textbf{99} (1) (1997), 1--24.

\bibitem {Biswas} A. Biswas, \textit{Solitary wave solution for the generalized
Kawahara equation}, Applied Mathematical Letters, \textbf{22} (2009), 208--210.

\bibitem {Bona1}J. L. Bona, S. M. Sun, B.-Y. Zhang,\textit{ A nonhomogeneous
boundary-value problem for the Korteweg-de Vries equation posed on a finite
domain}, Comm. Partial Differential Equations, \textbf{28} (2003), 1391--1436.

\bibitem {Bona2}J. L. Bona, S. M. Sun, B.-Y. Zhang, \textit{A non-homogeneous
boundary-value problem for the Korteweg-de Vries equation posed on a finite
domain II}, J. Differential Equations, \textbf{247} (9) (2009), 2558--2596.

\bibitem {Boyd}J. P. Boyd, \textit{Weakly non-local solitons for
capillary-gravity waves: fifth degree Korteweg-de Vries equation}, Phys. D,
\textbf{48} (1991), 129--146.

\bibitem {Bridges}T. Bridges, G. Derks, \textit{Linear instability of solitary
wave solutions of the Kawahara equation and its generalizations}, SIAM Journal
on Mathematical Analysis, \textbf{33} (6) (2002), 1356--1378.



\bibitem{CaKawahara}F. D. Araruna, R. A. Capistrano-Filho, G. G. Doronin, \textit{Energy decay for the modified Kawahara equation posed in a bounded domain}, J. Math. Anal. Appl., \textbf{385} (2) (2012), 743–756.

\bibitem{CaPaRo} R. A. Capistrano-Filho, A. F. Pazoto, L. Rosier, \textit{Internal controllability for the Korteweg-de Vries equation on a bounded domain}, ESAIM Control Optimization and Calculus Variations, \textbf{21} (2015), 1076--1107 

\bibitem{CaZh}  M. A. Caicedo, R. A. Capistrano-Filho, B.-Y. Zhang, \textit{Neumann boundary controllability of the Korteweg-de Vries equation on a bounded domain}, SIAM J. Control Optim., \textbf{55}  6 (2017) ,3503--3532.

\bibitem{CaPaRo1} R. A. Capistrano-Filho, A. F. Pazoto, L. Rosier, \textit{Control of Boussinesq system of KdV-KdV type on a bounded interval}, ESAIM Control Optimization and Calculus Variations \textbf{25} (2019), 1-55.

\bibitem {cerpa}E. Cerpa, {\em Exact controllability of a nonlinear Korteweg-de
Vries equation on a critical spatial domain}, SIAM J. Control Optim., \textbf{46} (2007), 877--899

\bibitem {cerpa1}E. Cerpa, E. Cr\'{e}peau, {\em Boundary controllability for
the nonlinear Korteweg-de Vries equation on any critical domain}, Ann. I.H.
Poincar\'{e}, \textbf{26} (2009), 457--475.

\bibitem {Cocr}J.-M. Coron, J.-M., E. Cr\'{e}peau, {\em Exact boundary
controllability of a nonlinear KdV equation with a critical length}, J. Eur.
Math. Soc., \textbf{6} (2004), 367--398.

\bibitem {MoChen} M. Chen, {\em Internal controllability of the Kawahara equation on a bounded domain}, Nonlinear Analysis, \textbf{185} (2019), 356--373.

\bibitem {Cui} S. B. Cui, D. G. Deng, S. P. Tao, \textit{Global existence of
solutions for the Cauchy problem of the Kawahara equation with $\mathit{L}%
_{\mathit{2}}$ initial data}, Acta Math. Sin. (Engl. Ser.), \textbf{22} (2006),
1457--1466. 

\bibitem {Doronin2}G. G. Doronin, N. A. Larkin, \textit{Kawahara equation in a
bounded domain}, Discrete Contin. Dyn. Syst., \textbf{10} (4) (2008), 783--799.

\bibitem{Doronin}G. G. Doronin, F. Natali, \textit{Exponential decay for a locally damped fifth-order equation posed on the line}, Nonlinear Analysis: Real World Applications, \textbf{30} (2016), 59--72.

\bibitem {Faminskii}A. V. Faminskii, N. A. Larkin, \textit{Initial-boundary
value problems for quasilinear dispersive equations posed on a bounded
interval,} Electron. J. Differential Equations,  \textbf{1} (2010), 1-20.

\bibitem{FaOp}A. V. Faminskii, M. A. Opritova, \textit{On the initial-boundary-value problem in a half-strip for a generalized Kawahara equation}, 	Journal of Mathematical Sciences,   \textbf{206} 1 (2015), 17–38.

\bibitem{Fa}A. V. Faminskii, \textit{Control Problems with an Integral Condition
for Korteweg--de Vries Equation on Unbounded Domains,} Journal of Optimization Theory and Applications,  \textbf{180} (2019), 290--302

\bibitem{Fa1}A. V. Faminskii, \textit{Controllability Problems for the Korteweg--de Vries
Equation with Integral Overdetermination}, Differential Equations, \textbf{55}  1 (2019), 126--137.

\bibitem {GG}O. Glass, S. Guerrero, {\em On the controllability of the fifth-order Korteweg–de Vries equation}, Ann. I. H. Poincaré, \textbf{26} (2009), 2181–-2209.

\bibitem {GGa}O. Glass, S. Guerrero, {\em Some exact controllability results
for the linear KdV equation and uniform controllability in the zero-dispersion
limit}, Asymptot. Anal., \textbf{60}  no. 1-2 (2008), 61--100.

\bibitem {GG1}O. Glass, O., S. Guerrero, {\em Controllability of the Korteweg-de
Vries equation from the right Dirichlet boundary condition}, Systems Control
Lett., \textbf{59} 7 (2010), 390--395.

\bibitem {goubet}O. Goubet, J. Shen, {\em On the dual Petrov-Galerkin
formulation of the KdV equation on a finite interval}, Adv. Differential
Equations \textbf{12} 2 (2007), 221--239.

\bibitem{Guzman} P. Guzm\'an, J.  Zhu,\textit{ Exact boundary controllability of a microbeam model}, Journal of Mathematical Analysis and Applications, \textbf{425} 2 (2015), 655-665.


\bibitem {Hunter}J. K. Hunter, J. Scheurle, \textit{Existence of perturbed
solitary wave solutions to a model equation for water waves}, Physica D,
\textbf{32} (1998), 253--268.

\bibitem {Iguchi}T. Iguchi, \textit{A long wave approximation for
capillary-gravity waves and the Kawahara Equations}, Academia Sinica (New
Series), \textbf{2} (2007), 179--220.

\bibitem {Jin} L. Jin, \textit{Application of variational iteration method and
homotopy perturbation method to the modified Kawahara equation}, Mathematical
and Computer Modelling, \textbf{49} (2009), 573--578.

\bibitem {Kakutani}T. Kakutani, \textit{Axially symmetric stagnation-point
flow of an electrically conducting fluid under transverse magnetic field}, J.
Phys. Soc. Japan, \textbf{15} (1960), 688--695.

\bibitem {Kaya} D. Kaya, K. Al-Khaled, \textit{A numerical comparison of a
Kawahara equation}, Phys. Lett. A, \textbf{363} (5-6) (2007), 433--439. 


\bibitem {Kawahara}T. Kawahara, \textit{Oscillatory solitary waves in
dispersive media}, J. Phys. Soc. Japan, \textbf{33} (1972), 260--264.

\bibitem{kepove} C. Kenig, G. Ponce, L. Vega, \textit{Higher-order nonlinear dispersive equations}, Proc. Amer. Math. Soc., \textbf{122} (1) (1994), 157-166.

\bibitem{khanal} N. Khanal, J. Wu, J.-M. Yuan, \textit{The Kawahara equation in weighted Sobolev space}, Nonlinearity, \textbf{21} (2008), 1489--1505.

\bibitem {Lions}J.-L. Lions,   {\em Sur les probl\'emes aux limites du type d'eriv'ee oblique}, Ann. Math. \textbf{64} (1956), 207-39l

\bibitem {Lions1}J.-L. Lions, {\em Exact controllability, stabilization and
perturbations for distributed systems}, SIAM Review, \textbf{30} (1988), 1--68.

\bibitem {Polat}N. Polat, D. Kaya, H.I. Tutalar, A\textit{n analytic and
numerical solution to a modified Kawahara equation and a convergence analysis
of the method}, Appl. Math. Comput., \textbf{179} (2006), 466--472.

\bibitem {Pomeau}Y. Pomeau, A. Ramani, B. Grammaticos, \textit{Structural
stability of the Korteweg-de Vries solitons under a singular perturbation},
Physica D, \textbf{31} (1988), 127--134.


\bibitem{ponce} G. Ponce, \textit{Lax pairs and higher order models for water waves}, J. Differential Equations, \textbf{102} (2) (1993), 360-381.

\bibitem {Rosier}L. Rosier, {\em Exact boundary controllability for the
Korteweg-de Vries equation on a bounded domain}, ESAIM Control Optim. Cal.
Var., \textbf{2} (1997), 33--55.

\bibitem {Shuangping}T. Shuangping, C. Shuangbin, \textit{Existence and
uniqueness of Solutions to nonlinear Kawahara equations}, Chinese Annals of
Mathematics, Series A, \textbf{23} (2) (2002), 221--228.

\bibitem {Sirendaoreji}Sirendaoreji, \textit{New exact travelling wave
solutions for the Kawahara and modified Kawahara equations. Chaos, Solitons
and Fractals}, \textbf{19} (1) (2004), 147--150.

\bibitem {Wazwaz1}A. M. Wazwaz, \textit{New solitary wave solutions to the
modified Kawahara equation}, Phys. Lett. A, \textbf{8} (4-5) (2007), 588--592.

\bibitem {Yusufoglu} E. Yusufoglu, A. Bekir, M. Alp, \textit{Periodic and
solitary wave solutions of Kawahara and modified Kawahara equations by using
Sine-Cosine method}, Chaos, Solitons and Fractals, \textbf{37} (2008),
1193--1197. 

\bibitem {DZhang}D. Zhang, \textit{Doubly periodic solutions of the modified
Kawahara equatio}n, Chaos Solitons Fractals, \textbf{25} (2005), 1155--1160.

\bibitem{Vasconcellos} A. L. C. Dos Santos, P. N. Da Silva, C. F. Vasconcellos, \textit{Entire functions related to stationary solutions of the kawahara equation}, Electronic Journal of Differential Equations, \textbf{43} (2016), 13pp.

\bibitem{Zeidler} E. Zeidler, Nonlinear functional analysis and its applications I. \textit{Springer-Verlag}, New York (1986).

\bibitem{zhang}B.-Y Zhang, X. Zhao, \textit{Control and stabilization of the Kawahara equation on a periodic domain}, Communications in Information and Systems, \textbf{12} (1) (2012), 77--96.

\bibitem{zhang1}B.-Y Zhang, X. Zhao, \textit{Global controllability and stabilizability of Kawahara equation on a periodic domain}, Mathematical Control \& Related Fields, \textbf{5} 2 (2015),  335--358.

\end{thebibliography}
\end{document}